\documentclass[11pt,reqno,a4paper]{amsart}

\usepackage{kantlipsum}
\calclayout

\usepackage{amsmath,amssymb,amscd}
\usepackage{upgreek}
\usepackage{wasysym,footnote}
\usepackage[all]{xy}
\input{genrank.sty}

\newtheorem{theorem}{Theorem}[section]
\newtheorem{lemma}[theorem]{Lemma}
\newtheorem{corollary}[theorem]{Corollary}
\newtheorem{proposition}[theorem]{Proposition}

\theoremstyle{definition}
\newtheorem{definition}[theorem]{Definition}

\newtheorem{example}[theorem]{Example}
\newtheorem{conjecture}[theorem]{Conjecture}

\theoremstyle{remark}
\newtheorem{remark}[theorem]{Remark}

\newtheorem{question}[theorem]{Question}

\numberwithin{equation}{section}

\DeclareMathOperator{\add}{add}
\DeclareMathOperator{\Aut}{Aut}
\DeclareMathOperator{\coker}{coker}

\DeclareMathOperator{\E}{E}

\DeclareMathOperator{\Ext}{Ext}
\DeclareMathOperator{\ext}{ext}
\DeclareMathOperator{\GL}{GL}

\DeclareMathOperator{\Hom}{Hom}

\DeclareMathOperator{\img}{im}
\DeclareMathOperator{\ind}{ind}

\DeclareMathOperator{\PHom}{PHom}
\DeclareMathOperator{\IHom}{IHom}
\DeclareMathOperator{\THom}{THom}
\DeclareMathOperator{\rank}{rank}
\DeclareMathOperator{\red}{red}
\DeclareMathOperator{\rep}{rep}

\DeclareMathOperator{\Tr}{Tr}
\DeclareMathOperator{\PC}{PC}

\newcommand{\Ec}{{\check{\E}}}
\newcommand{\ec}{{\check{\e}}}

\newcommand{\fc}{{\check{f}}}
\newcommand{\dc}{{\check{d}}}

\newcommand{\betac}{{\check{\beta}}}

\newcommand{\dtc}{{\check{\delta}}}
\newcommand{\epc}{{\check{\ep}}}
\newcommand{\lc}{{\check{l}}}

\newcommand{\rc}{{\check{r}}}

\newcommand{\e}{{\rm e}}

\newcommand{\g}{{\sf g}}

\renewcommand{\S}{{\mc{S}}}
\newcommand{\op}[1]{\operatorname{#1}}
\newcommand{\mb}[1]{\mathbb{#1}}
\newcommand{\mc}[1]{\mathcal{#1}}

\renewcommand{\b}[1]{\bold{#1}}

\newcommand{\ep}{{\epsilon}}

\newcommand{\proj}{\operatorname{proj}\text{-}}

\newcommand{\br}[1]{\overline{#1}}

\newcommand{\innerprod}[1]{\langle#1\rangle}
\renewcommand{\ss}[2]{^{{#1}\cdot{\rm #2}}}

\renewcommand{\b}[1]{\bold{#1}}

\newcommand{\dv}{\underline{\dim}}

\newcommand{\wtd}[1]{\widetilde{#1}}

\newcommand{\ckQ}{\widehat{kQ}}

\newcommand{\CQ}{\mc{C}_{Q,\mc{S}}}
\newcommand{\T}{\b{T}}
\newcommand{\C}{\mc{C}}

\newcommand{\mub}{\mu_{\b{u}}}

\newcommand{\tauh}{\hat{\tau}}
\renewcommand{\mod}{\operatorname{mod}}

\begin{document}
	
\title{On the General Ranks of QP Representations}
\author{JiaRui Fei}
\address{School of Mathematical Sciences, Shanghai Jiao Tong University, 800 DongChuan Road, Shanghai, China}
\email{jiarui@sjtu.edu.cn}
\thanks{The author was supported in part by National Natural Science Foundation of China (No. 12131015 and No. 11971305)}

\subjclass[2020]{Primary 16G10; Secondary 13F60}

\date{}
\dedicatory{}
\keywords{General Rank, Representation, Quiver with Potential, Mutation, Jacobian Algebra}

\begin{abstract} We propose a mutation formula for the general rank from a principal component $\PC(\delta)$ of representations to another one $\PC(\ep)$ for a quiver with potential.
We give sufficient conditions for the formula to hold. In particular, the formula holds when any of $\delta$ and $\ep$ is reachable.
We discover several related mutation invariants.
\end{abstract}
\maketitle

\section{Introduction}
Schofield introduced the general rank for quiver representations in his theory of general representations \cite{S}. It was further pursued by Crawley-Boevey in \cite{CB}.
Let $Q$ be a finite quiver without oriented cycles, and $\alpha,\beta$ be some dimension vectors of $Q$.
There is an open subset $U$ of $\rep_\alpha(Q)\times \rep_\beta(Q)$ and a dimension vector $\gamma$ such that for all $(M,N)\in U$, $\Hom_Q(M,N)$ has the minimal dimension and
$\{\phi\in \Hom_Q(M,N) \mid \rank\phi=\gamma\}$ is open and non-empty in $\Hom_Q(M,N)$.
The dimension vector $\gamma$ is called the {\em general rank} from $\rep_\alpha(Q)$ to $\rep_\beta(Q)$.

A remarkable property of $\gamma$ is the following 
\begin{equation}\label{eq:ext} \ext_Q(\alpha,\beta)=-\innerprod{\alpha-\gamma,\beta-\gamma}_Q = \ext_Q(\alpha-\gamma,\beta-\gamma).
\end{equation}
where $\ext_Q(\alpha,\beta)$ is the generic (minimal) value of $\dim\Ext_Q^1(M,N)$ on $\rep_\alpha(Q)\times \rep_\beta(Q)$, and $\innerprod{-,-}_Q$ is the Euler form of $Q$.
Schofield further used this property to deduce an algorithm to compute $\ext_Q(\alpha,\beta)$.
This algorithm is one of the key ingredients in Derksen-Weyman's proof of the {\em saturation conjecture} \cite{DW1}.
The algorithm was reformulated in terms of {\em tropical $F$-polynomials} and generalized to finite-dimensional algebras in \cite{Ft}.

Fock and Goncharov formulated {\em their duality pairing conjecture} for cluster varieties in \cite{FG}.
We gave their duality pairing a representation-theoretic interpretation in \cite{Ft}.
We showed that the duality pairing conjecture holds for generic bases if $f_{\epc}(\delta) = \fc_{\delta}(\epc)$ holds for any pair $(\delta,\epc)$ of weight vectors in $\mb{Z}^{Q_0}$,
where $f_{\epc}$ and $\fc_{\delta}$ are the generic tropical $F$-polynomials (see Appendix for the definitions) of a relevant quiver with potential $(Q,\S)$.
As pointed out in \cite{Ft}, more optimistically one may expect the following {\em generic pairing} 
$$f_{\epc}(\delta) =\hom(\delta,\epc) = \fc_{\delta}(\epc)\ \text{ for any pair $(\delta,\epc)$}. $$
This turns out to be equivalent to the saturation conjecture for nondegenerate quivers with potentials (see Definition \ref{D:saturation} and Proposition \ref{P:saturation}).

One would hope that by studying the general ranks of representations of quivers with potentials, the duality pairing conjecture and the saturation conjecture for nondegenerate QPs can be settled. This is our original motivation for this study.
Let us first explain what we mean by the general ranks in the setting of quivers with potentials.
The definition of the general rank has a straightforward generalization (see Lemma \ref{L:genrank}) if we replace the representation spaces of a quiver $Q$ by some irreducible components in the representation variety of any basic algebra, in particular the Jacobian algebra $J$ of $(Q,\S)$ \cite{DWZ1}.
For Jacobi-finite quivers with potentials, the irreducible components of interests for us are those hit by
general presentations \cite{DF}.
More precisely, for any $\delta\in\mb{Z}^{Q_0}$ we define the presentation space
$$\PHom_J(\delta):=\Hom_J(P([-\delta]_+),P([\delta]_+)).$$
Here $P(\beta) = \bigoplus_{u\in Q_0} \beta(u) P_u$ and $P_u$ is the indecomposable projective representation corresponding to $u$.
The vector $\delta$ is called the weight vector or the $\delta$-vector of the presentation space.
There is an open subset $U$ of $\PHom_J(\delta)$ such that the cokernels of presentations in $U$ lie in an irreducible component $\PC(\delta)$ of the representation variety of $J$ (see Definition \ref{D:PC}). 
As shown in \cite{P}, such a component is exactly a {\em strongly reduced component} introduced in \cite{GLS}.
The general rank we consider in this article is the one from the component $\PC(\delta)$ to another one $\PC(\ep)$, denoted by $\rank(\delta,\ep)$.

It is convenient to introduce the {\em decorated} representations \cite{DWZ1}, which are in bijection with the projective presentations up to homotopy equivalence.
A {decorated representation} $\mc{M}=(M,M^-)$ is called general of weight $\delta$ if it corresponds to a general presentation of weight $\delta$. In this case, $M$ is a general representation in $\PC(\delta)$.
By abuse of language, $M$ is also called a general representation of weight $\delta$.

However, we find it difficult to prove a variation of \eqref{eq:ext} in the setting of quivers with potentials (see Question \ref{q:var}).
Most geometric arguments in Schofield's proof break down here, which seems to be the largest obstacle in generalizing Schofield's theory. 
On the other hand, we get some new weapon from the theory of cluster algebras and quivers with potentials.
Thanks to Derksen-Weyman-Zelevinsky's mutation, we are able to find a way for computing the general ranks. Our method is expected to work for any pair $(\delta,\ep)$ but we can only show it works under some conditions. Those conditions are trivially satisfied if one of $\delta$ and $\ep$ is {\em extended-reachable}, that is, can be obtained by a sequence of mutations and $\tau$.
The sufficient conditions are closely related to a conjecture, which implies the duality pairing conjecture and the saturation conjecture for nondegenerate QPs.

For any weight vector $\ep$ of $(Q,\S)$, we define a pair of operators, which plays an important role not only in our method but also in the crystal structure of upper cluster algebras \cite{Fc}.
\begin{definition} For any weight vector $\ep$ of $(Q,\S)$, we define the two operators $r_{\ep}$ and $l_{\ep}$ on the set of weight vectors of $(Q,\S)$ as follows:
\begin{align*} 	r_{\ep} (\delta) &= \delta+\ep +\rank(\ep, \tau\delta) B; \\
	l_{\ep} (\delta) &= \delta-\epc +\rank(\delta,\ep) B,
\end{align*}
where $\tau$ (defined before Theorem \ref{T:genpi}) is related to the AR-translation and $B$ is the skew-symmetric matrix attached to $Q$.
\end{definition}

By the mutation $\mu_u$ of $\delta$ for $u\in Q_0$, we mean the formula \eqref{eq:gmu}.
It is an important problem to determine when the operators $r_\ep$ and $l_\ep$ commute with mutations, that is, 
\begin{equation}\label{intro:commeq} \mu_u(r_\ep(\delta)) = r_{\mu_u(\ep)}(\mu_u(\delta)) \ \text{ and }\ \mu_u(l_\ep(\delta)) = l_{\mu_u(\ep)}(\mu_u(\delta)).
\end{equation}
To motivate this problem, we see that if \eqref{intro:commeq} holds for any sequence of mutations, then it is possible to recursively compute $\rank(\ep, \tau\delta)$ or $\rank(\delta,\ep)$ (see the discussion after Example \ref{ex:nonpair} for details). 
In fact, we obtain the following explicit mutation formula for the general rank.
\begin{theorem}[Theorem \ref{T:murank}] \label{T:intro2} Let $\gamma_r = \rank(\ep,\tau\delta)$ and $\gamma_l = \rank(\delta,\ep)$.	We denote $\gamma_r' = \rank(\ep',\tau\delta')$ and $\gamma_l' = \rank(\delta',\ep')$, where $\delta'=\mu_u(\delta)$ and $\ep'=\mu_u(\ep)$. 
	\begin{enumerate}
		\item If $\mu_u (r_\ep(\delta))=r_{\ep'}(\delta')$, then
		\begin{align*} 
			\gamma_r'(v) &= \begin{cases} \gamma_r(v) &  \text{for all $v\neq u$}, \\
				\gamma_r [b_u]_+ - \gamma_r(u) + [\delta(u)]_+ + [\ep(u)]_+ - [r_\ep(\delta)(u)]_+ & \text{for all $v=u$}.
			\end{cases}
		\end{align*}
		\item If $\mu_u (l_\ep(\delta))=l_{\ep'}(\delta')$, then\begin{align*}
			\gamma_l'(v) &= \begin{cases} \gamma_l(v) & \text{for all $v\neq u$}, \\
				\gamma_l [b_u]_+ - \gamma_l(u) + [\delta(u)]_+ + [-\epc(u)]_+ - [l_\ep(\delta)(u)]_+ & \text{for all $v=u$}.
		\end{cases} \end{align*}
	\end{enumerate}
\end{theorem}

We say $(\delta,\ep)$ has {\em completely extremal rank} if any of the following occurs: 
$$\rank(\delta,\ep)=0,\quad  \rank(\delta,\ep)=\dv(\delta),\quad \rank(\delta,\ep)=\dv(\ep),$$
where $\dv(\delta)$ denotes the dimension vector of a general representation in $\PC(\delta)$.
It turns out that if there is a sequence of mutations $\mub$ such that $(\mub(\delta),\mub(\ep))$ has completely extremal rank, then $r_\ep$ and $l_\ep$ commute with any sequence of mutations and $\tau^i$ (Proposition \ref{P:rlmu}).

\begin{conjecture}[Conjecture \ref{c:cer}] \label{intro:cer} Let $(Q,\S)$ be a nondegenerate Jacobi-finite QP. For any pair $(\delta,\ep)$ of $\delta$-vectors of $(Q,\S)$, there is a sequence of mutations $\mub$ such that $(\mub(\delta),\mub(\ep))$ has completely extremal rank.
\end{conjecture}

Another main result is the following incarnation of the raising and lowering operators in exact sequences. 
The proof relies on a stronger implication of $(\delta,\ep)$ having extremal ranks,
namely, $r_\ep$ (or $l_\ep$) can be {\em generically lifted at $\delta$} (Lemmas \ref{L:rlmu}, \ref{L:rltri} and \ref{L:rltri1}).

\begin{theorem}[Theorem \ref{T:rl}] \label{T:intro1} Let $\ep$ and $\delta$ be two weight vectors for a quiver with potential $(Q,\S)$.
\begin{enumerate}
	\item Suppose that there is an extended mutation sequence $\mub$ such that $(\mub(\ep), \mub(\tau\delta))$ has completely extremal rank. Then there is an exact sequence
	\begin{equation*} \cdots \to  \tauh^{-1} \mc{M}\xrightarrow{} \tauh^{-1} \mc{R} \xrightarrow{} \tauh^{-1} \mc{E}\xrightarrow{h_{-1}} M\xrightarrow{} R \xrightarrow{} E \xrightarrow{h_0} \tauh \mc{M} \xrightarrow{} \tauh \mc{R} \xrightarrow{} \tauh \mc{E} \xrightarrow{h_1} \tauh^2 \mc{M}\to \cdots,\end{equation*}	
	where $\mc{R}$ is general of weight $r_\ep(\delta)$, $(\mc{M},\mc{E})$ is general as a pair of weights $\delta$ and $\ep$, and $h_i$ is a general homomorphism in $\Hom_J(\tauh^{i}\mc{E}, \tauh^{i+1} \mc{M})$.
	\item Suppose that there is an extended mutation sequence $\mub$ such that $(\mub(\delta), \mub(\ep))$ has completely extremal rank. Then there is an exact sequence
	\begin{equation*} \cdots \to  \tauh^{-1} \mc{L}\xrightarrow{} \tauh^{-1} \mc{M} \xrightarrow{g_{-1}} \tauh^{-1} \mc{E}\xrightarrow{} L\xrightarrow{} M \xrightarrow{g_0} E \xrightarrow{} \tauh \mc{L} \xrightarrow{} \tauh \mc{M} \xrightarrow{g_1} \tauh \mc{E} \xrightarrow{} \tauh^2 \mc{L}\to \cdots,\end{equation*}	
	where $\mc{L}$ is general of weight $l_\ep(\delta)$, $(\mc{M},\mc{E})$ is general as a pair of weight $\delta$ and $\ep$, and $g_i$ is a general homomorphism in $\Hom_J(\tauh^{i} \mc{M}, \tauh^{i} \mc{E})$.
\end{enumerate}
\end{theorem}
\noindent Here, $\tauh^i \mc{M}$ denotes the representation obtained from $\tau^i \mc{M}$ by forgetting the decorated part.
$(\mc{E},\mc{M})$ is general as a pair of weight $\delta$ and $\ep$ means that $(E,M)$ can be chosen in an open subset of $\PC(\delta)\times\PC(\ep)$.


With the explicit formulas in Theorem \ref{T:intro2}, we find some interesting mutation-invariants assuming $r_\ep$ and $l_\ep$ can be generically lifted at $\delta$ (see Lemma \ref{L:rlmu}).
\begin{proposition}[Proposition \ref{P:muinv}] Suppose that $l_\ep$ and $r_\ep$ commute with the mutation $\mu_u$. Let $\gamma_r = \rank(\ep,\tau\delta)$ and $\gamma_l = \rank(\delta,\ep)$.  Then the following are invariant under $\mu_u$:
	\begin{align*} 
		h_l(\delta,\ep):=&\epc({\gamma}_l) - \hom(\delta,\ep) + \hom(l_\ep(\delta),\ep);  \\   
		e_r(\delta,\ep):=&\ep(\gamma_r)	-\e(\delta, \ep) +\e(r_{\ep}(\delta),\ep).  
	\end{align*}
Here $\hom(\delta,\ep)$ is the generic value of $\dim\Hom_J(M,N)$ for $(M,N)\in \PC(\delta)\times \PC(\ep)$, and $\e(\delta,\ep)$ is a certain generic value on $\PHom_J(\delta)\times \PHom_J(\ep)$ introduced in \cite{DF} (see also below Lemma \ref{L:homotopy}).
\end{proposition}

Note that the assumptions in Theorem \ref{T:intro1} are always satisfied if either $\ep$ or $\delta$ is extended-reachable.
In this case, Theorem \ref{T:intro1} and related results take a particularly nice form (Corollaries \ref{C:rlrigid} and \ref{C:rle}). Moreover, the operators $r_\ep$ and $l_\ep$ are inverse of each other (Proposition \ref{P:rlid}).
We also find an interesting property of these operators related to the canonical decomposition of presentations \cite{DF}.
We write $\ep = \ep_1\oplus \ep_2$ if a general presentation of weight $\ep$ is a direct sum of a presentation of weight $\ep_1$ and a presentation of weight $\ep_2$.

\begin{proposition}[Proposition \ref{P:CDP}] For any two weight vectors $\ep_1$ and $\ep_2$ of a QP, the following are equivalent: \begin{enumerate}
		\item $\ep = \ep_1 \oplus \ep_2$;
		\item $r_\ep = r_{\ep_1} r_{\ep_2} = r_{\ep_2} r_{\ep_1}$;
		\item $l_\ep = l_{\ep_1} l_{\ep_2} = l_{\ep_2} l_{\ep_1}$.
	\end{enumerate}
\end{proposition}

\subsection{Organization}
In Section \ref{S:QP} we briefly review the theory of quivers with potentials.
In Section \ref{S:GP} after a brief review of the theory of general presentations \cite{DF}, 
we prove that general presentations behave well under the Auslander-Reiten translation (Theorem \ref{T:genpi}), which we think is mildly new.
In Section \ref{S:2CY} we review the theory of $2$-Calabi-Yau triangulated categories developed by C. Amiot, B. Keller, Y. Palu, P. Plamondon and so on. Notable is the Palu's formula on index (Lemma \ref{L:index}), which is important to us.
In Section \ref{S:GR} we introduce the pair of operators $r_\ep$ and $l_\ep$, and prove our first main result (Theorem \ref{T:rl}). 
In Section \ref{S:Prop} we study the properties around the general ranks, including our second main result (Theorem \ref{T:murank}), and some mutation invariants (Proposition \ref{P:muinv}).
In Appendix we explain the relationship between the saturation conjecture for nondegenerate QPs and the Fock-Goncharov's duality pairing conjecture for the generic bases, and explain why Conjecture \ref{c:hev} implies them.
\subsection{About Notations}
By a quiver $Q$ we mean a quadruple $Q = (Q_0,Q_1, t, h)$ where $Q_0$ is a finite set of vertices, $Q_1$ is a finite set of arrows, 
and $t$ and $h$ are the tail and head functions $Q_1 \to Q_0$.
Let $k$ denote an algebraically closed field of characteristic zero, and $A\cong kQ/I$ be a basic finite-dimensional $k$-algebra.

All modules are right modules, and all vectors are row vectors.
For the direct sum of $n$ copies of $M$, we write $nM$ instead of the traditional $M^{\oplus n}$.
We write $\hom,\ext$ and $\e$ for $\dim\Hom, \dim\Ext$, and $\dim \E$. The superscript $*$ is the trivial dual for vector spaces.
\begin{align*}
	& \rep A && \text{the category of finite-dimensional representations of $A$} &\\
	& \rep_\alpha(A) && \text{the space of $\alpha$-dimensional representations of $A$} &\\
	& \proj A && \text{the category of finite-dimensional projective representations of $A$} &\\
	& S_u && \text{the simple representation supported on the vertex $u$} &\\
	& P_u, I_u && \text{the projective cover and the injective envelope  of $S_u$} &\\
	& \dv M && \text{the dimension vector of $M$} & \\
	& \tauh^i \mc{M} && \text{the representation obtained by forgetting the decorated part of $\tau^i \mc{M}$} & 
\end{align*}


\section{A Review on Representation Theory of Quivers with Potentials} \label{S:QP}
\subsection{Decorated Representations and Presentations} 
Let $Q$ be a finite quiver with no loops. For such a quiver, we associate a skew-symmetric matrix $B_Q$ given by
$$B_Q(u,v) = |\text{arrows $u\to v$}| - |\text{arrows $v\to u$}|.$$ 
Following \cite{DWZ1}, we define a potential $\mc{S}$ on a quiver $Q$ as a (possibly infinite) linear combination of oriented cycles in $Q$.
More precisely, a {\em potential} is an element of the {\em trace space} $\Tr(\ckQ):=\ckQ/[\ckQ,\ckQ]$,
where $\ckQ$ is the completion of the path algebra $kQ$ and $[\ckQ,\ckQ]$ is the closure of the commutator subspace of $\ckQ$.
The pair $(Q,\mc{S})$ is a {\em quiver with potential}, or QP for short.
For each arrow $a\in Q_1$, the {\em cyclic derivative} $\partial_a$ on $\widehat{kQ}$ is defined to be the linear extension of
$$\partial_a(a_1\cdots a_d)=\sum_{k=1}^{d}a^*(a_k)a_{k+1}\cdots a_da_1\cdots a_{k-1}.$$
For each potential $\mc{S}$, its {\em Jacobian ideal} $\partial \mc{S}$ is the closed (two-sided) ideal in $\ckQ$ generated by all $\partial_a \mc{S}$.
The {\em Jacobian algebra} $J=J(Q,\mc{S})$ is $\widehat{kQ}/\partial \mc{S}$.
A QP is {\em Jacobi-finite} if its Jacobian algebra is finite-dimensional.
All QPs in this article will be assumed to be Jacobi-finite.

\begin{definition} A decorated representation of the Jacobian algebra $J$ is a pair $\mc{M}=(M,M^-)$,
	where $M\in \rep J$, and $M^-$ is a finite-dimensional $k^{Q_0}$-module.
\end{definition}
\noindent 	By abuse of language, we also say that $\mc{M}$ is a representation of $(Q,\mc{S})$.
When appropriate, we will view an ordinary representation $M$ as the decorated representation $(M,0)$.

Following \cite{DF} we call a homomorphism between two projective representations, a {\em projective presentation} (or presentation in short). 
As a full subcategory of the category of complexes in $\rep J$, the category of projective presentations is Krull-Schmidt as well.
Sometimes it is convenient to view a presentation $P_-\to P_+$ as elements in the homotopy category $K^b(\proj J)$ of bounded complexes of projective representations of $J$.
Our convention is that $P_-$ sits in degree $-1$ and $P_+$ sits in degree $0$.

We denote by $P_u$ (resp. $I_u$) the indecomposable projective (resp. injective) representation of $J$ corresponding to the vertex $u$ of $Q$. 
For $\beta \in \mb{Z}_{\geq 0}^{Q_0}$ we write $P(\beta)$ for $\bigoplus_{u\in Q_0} \beta(u)P_u$ and $I(\beta)$ for $\bigoplus_{u\in Q_0} \beta(u)I_u$. 
\begin{definition}\footnote{The $\delta$-vector is the same one defined in \cite{DF}, but is the negative of the $\g$-vector defined in \cite{DWZ2}. }
	The {\em $\delta$-vector} (or {\em weight vector}) of a presentation 
	$$d: P(\beta_-)\to P(\beta_+)$$
	is the difference $\beta_+-\beta_- \in \mb{Z}^{Q_0}$.
	When working with injective presentations 
	$$\dc: I(\betac_+)\to I(\betac_-),$$
	we call the vector $\betac_+ - \betac_-$ the {\em $\check{\delta}$-vector} of $\dc$.
\end{definition}
\noindent The $\delta$-vector is just the corresponding element in the Grothendieck group of $K^b(\proj J)$.

Let $\nu$ be the Nakayama functor $\Hom_J(-,J)^*$.
There is a map still denoted by $\nu$ sending a projective presentation to an injective one
$$P_-\to P_+\ \mapsto\ \nu(P_-) \to \nu(P_+).$$
Note that if there is no direct summand of the form $P_i\to 0$, then $\ker(\nu d) = \tau\coker(d)$ where $\tau$ is the classical Auslander-Reiten translation. 

Let $\mc{R}ep(J)$ be the set of decorated representations of $J$ up to isomorphism, and $K^2(\proj J)$ be the full subcategory of complexes of length 2 in $K^b(\proj J)$. There is a bijection between the additive categories $\mc{R}ep(J)$ and $K^2(\proj J)$ mapping any representation $M$ to its minimal presentation in $\rep J$, and the simple representation $S_u^-$ of $k^{Q_0}$ to $P_u\to 0$.
Now we can naturally extend the classical AR-translation to decorated representations:
$$\xymatrix{\mc{M} \ar[r]\ar@{<->}[d] & \tau \mc{M} \ar@{<->}[d] \\ d_{\mc{M}} \ar[r] & \nu(d_{\mc{M}})}$$
Note that this definition agrees with the one in \cite{DF}. We will find it convenient to introduce the notation $\tauh \mc{M}$ to denote the representation obtained from $\tau \mc{M}$ by forgetting the decorated part.

Suppose that $\mc{M}$ corresponds to a projective presentation
$d_{\mc{M}}: P(\beta_-)\to P(\beta_+)$. 
The Betti vectors $\beta_{-,\mc{M}}$ and $\beta_{+,\mc{M}}$ are by definition $\beta_-$ and $\beta_+$ respectively.
The $\delta$-vector $\delta_{\mc{M}}$ of $\mc{M}$ is by definition the $\delta$-vector of $d_{\mc{M}}$, that is, $\beta_+-\beta_-$.
If working with the injective presentations, we can define the $\dtc$-vector $\dtc_{\mc{M}}$ of $\mc{M}$.
It is known \cite{DWZ2} that $\delta_{\mc{M}}$ and $\dtc_{\mc{M}}$ are related by
\begin{equation}\label{eq:delta2dual} \dtc_{\mc{M}}  =  \delta_{\mc{M}} + (\dv M) B_Q. \end{equation}


\begin{definition}[{\cite{DWZ2,DF}}] \label{D:HomE} Given any projective presentation $d: P_-\to P_+$ and any $N\in \rep(A)$, we define $\Hom(d,N)$ and $\E(d,N)$ to be the kernel and cokernel of the induced map:
	\begin{equation} \label{eq:HE} 0\to \Hom(d,N)\to \Hom_J(P_+,N) \xrightarrow{} \Hom_J(P_-,N) \to \E(d, N)\to 0.
	\end{equation}
	Similarly for an injective presentation $\dc: I_+\to I_-$, we define $\Hom(M,\dc)$ and $\Ec(M,\dc)$ to be the kernel and cokernel of the induced map $\Hom_J(M,I_+) \xrightarrow{} \Hom_J(M,I_-)$.
	It is clear that 
	$$\Hom(d,N) = \Hom_J(\coker(d),N)\ \text{ and }\ \Hom(M,\dc) = \Hom_J(M,\ker(\dc)).$$
We set $\Hom_J(\mc{M},\mc{N})=\Hom(d_{\mc{M}},N)=\Hom(M,\dc_{\mc{N}})$, 
$\E_J(\mc{M},\mc{N}) := \E(d_{\mc{M}},N)$ and $\Ec_J(\mc{M},\mc{N}) := \Ec(M,\dc_{\mc{N}})$.
\end{definition}	
\noindent Note that according to this definition, we have that $\Hom_J(\mc{M},\mc{N}) = \Hom_J(M,N)$.\footnote{This definition is slightly different from the one in \cite{DWZ2}, which involves the decorated part.}
We also set $\E(d_{\mc{M}},d_{\mc{N}}) = \E_J(\mc{M},\mc{N})$ and $\Ec(\dc_{\mc{M}},\dc_{\mc{N}}) = \Ec_J(\mc{M},\mc{N})$.
We refer readers to \cite{DF} for an interpretation of $\E_J(\mc{M},\mc{N})$ in terms of the presentations $d_{\mc{M}}$ and $d_{\mc{N}}$.
We call $\mc{M}$ or $d_{\mc{M}}$ {\em rigid} if $\E_J(\mc{M},\mc{M})=0$.

\begin{lemma}[{\cite[Corollary 10.8 and Proposition 7.3]{DWZ2}, \cite[Corollary 7.6]{DF}}]  \label{L:H2E} We have the following equalities:
	\begin{enumerate}
		\item{} $\E_J(\mc{M},\mc{N})=\Hom_J(\mc{N},\tau\mc{M})^*\text{ and }\Ec_J(\mc{M},\mc{N})=\Hom_J(\tau^{-1}\mc{N},\mc{M})^*.$
		\item{} $\E_J(\mc{M},\mc{M})=\Ec_J(\mc{M},\mc{M})=\E_J(\tau\mc{M},\tau\mc{M})$.
	\end{enumerate}
\end{lemma}

\subsection{Mutation of QPs}
In \cite{DWZ1} and \cite{DWZ2}, the mutation of quivers with potentials is invented to model the cluster algebras.
Let $u$ be a vertex of $Q$ away from all $2$-cycles. The {\em mutation} $\mu_u$ of a QP $(Q,\mc{S})$ at $u$ is defined as follows.
The first step is to define the following new QP $\wtd{\mu}_u(Q,\mc{S})=(\wtd{Q},\wtd{\mc{S}})$.
We put $\wtd{Q}_0=Q_0$ and $\wtd{Q}_1$ is the union of three different kinds
\begin{enumerate}
	\item[$\bullet$] all arrows of $Q$ not incident to $u$,
	\item[$\bullet$] a composite arrow $[ab]$ from $t(a)$ to $h(b)$ for each $a,b$~with~$h(a)=t(b)=u$,
	\item[$\bullet$] an opposite arrow $a^\star$ (resp. $b^\star$) for each incoming arrow $a$ (resp. outgoing arrow $b$) at $u$.
\end{enumerate}
The new potential on $\wtd{Q}$ is given by
$$\wtd{\mc{S}}:=[\mc{S}]+\sum_{h(a)=t(b)=u}b^\star a^\star[ab],$$
where $[\mc{S}]$ is obtained by substituting $[ab]$ for each word $ab$ occurring in $\mc{S}$. Finally we define $(Q',\mc{S}')=\mu_u(Q,\mc{S})$ as the {\em reduced part} (\cite[Definition 4.13]{DWZ1}) of $(\wtd{Q},\wtd{\mc{S}})$.
For this last step, we refer readers to \cite[Section 4,5]{DWZ1} for details.
A sequence of vertices is called {\em admissible} for $(Q,\S)$ if its mutation along this sequence is defined. 
If all sequences are admissible for $(Q,\S)$ then we call $(Q,\S)$ {\em nondegenerate}.

Now we start to define the mutation of decorated representations of $J:=J(Q,\mc{S})$.
Consider the triangle of linear maps as in \cite{DWZ1} with $\beta_u \gamma_u=0$ and $\gamma_u \alpha_u=0$.
$$\vcenter{\xymatrix@C=5ex{
		& M(u) \ar[dr]^{\beta_u} \\
		\bigoplus_{h(a)=u} M(t(a)) \ar[ur]^{\alpha_u} && \bigoplus_{t(b)=u} M(h(b)) \ar[ll]^{\gamma_u} \\
}}$$

We first define a decorated representation $\wtd{\mc{M}}=(\wtd{M},\wtd{M}^-)$ of $\wtd{\mu}_u(Q,\mc{S})$.
We set \begin{align*}
&\wtd{M}(v)=M(v),\quad  \wtd{M}^-(v)=M^-(v)\quad (v\neq u); \\
&\wtd{M}(u)=\frac{\ker \gamma_u}{\img \beta_u}\oplus \img \gamma_u \oplus \frac{\ker \alpha_u}{\img \gamma_u} \oplus M^-(u),\quad \wtd{M}^-(u)=\frac{\ker \beta_u}{\ker \beta_u\cap \img \alpha_u}.
\end{align*}
We then set $\wtd{M}(a)=M(a)$ for all arrows not incident to $u$, and $\wtd{M}([ab])=M(ab)$.
It is defined in \cite{DWZ1} a choice of linear maps
$\wtd{M}(a^\star)$ and $\wtd{M}(b^\star)$ making
$\wtd{M}$ a representation of $(\wtd{Q},\wtd{\mc{S}})$.
We refer readers to \cite[Section 10]{DWZ2} for details.
Finally, we define $\mc{M}'=\mu_u(\mc{M})$ to be the {\em reduced part} (\cite[Definition 10.4]{DWZ1}) of $\wtd{\mc{M}}$.

We say a decorated representation $\mc{M}$ of $(Q,\mc{S})$ is {\em negative reachable} if there is a sequence of mutations $\mub=(\mu_{u_1},\cdots,\mu_{u_l})$ such that $\mub(\mc{M})$ is {\em negative}, i.e., $\mub(\mc{M})$ has only the decorated part.
Similarly we say $\mc{M}$ is {\em positive reachable} if there is a sequence of mutations $\mub$ such that $\mub(\mc{M})$ is a projective representation.

Let us recall several formula relating the $\delta$-vector of $\mc{M}$ and its mutation $\mu_{u}(\mc{M})$. We will use the notation $[b]_+$ for $\max(b,0)$, and set $b_{u,v}:=B_Q(u,v)$.
\begin{lemma}[{\cite[Lemma 5.2]{DWZ2}}] \label{L:gdmu} Let $\delta=\delta_{\mc{M}}$ and $\delta'=\delta_{\mub(\mc{M})}$.
We use the similar notation for $\dtc=\dtc_{\mc{M}}$ and the dimension vectors $d=\dv(M)$. Then 
\begin{align} \delta'(v) &= \begin{cases} -\delta(u) & \text{if $v=u$}\\ \delta(v) - [-b_{u,v}]_+\beta_-(u) + [b_{u,v}]_+\beta_+(u) & \text{if $v\neq u$.} \end{cases}\\
\dtc'(v) &= \begin{cases} -\dtc(u) & \text{if $v=u$}\\ \dtc(v) - [b_{u,v}]_+\betac_-(u) + [-b_{u,v}]_+\betac_+(u) & \text{if $v\neq u$.} \end{cases}
\end{align} 
\end{lemma}
\noindent We remark that the mutated $\delta$-vector $\delta'$ is {\em not} completely determined by $\delta$ (we need $\beta_-$ and $\beta_+$). But see also Remark \ref{r:genmu}.

\begin{lemma}\label{L:HEmu} {\em \cite[Proposition 6.1, and Theorem 7.1]{DWZ2}} Let $J'=J(\mub(Q,\S))$, $\mc{M}'=\mu_u(\mc{M})$ and $\mc{N}'=\mu_u(\mc{N})$. 
We have that \begin{enumerate}
\item{} $\hom_{J'}(\mc{M}',\mc{N}')-\hom_J(\mc{M},\mc{N})=\beta_{-,\mc{M}}(u)\betac_{-,\mc{N}}(u)-\beta_{+,\mc{M}}(u)\betac_{+,\mc{N}}(u)$;
\item{} $\e_{J'}(\mc{M}',\mc{N}')-\e_J(\mc{M},\mc{N})=\beta_{+,\mc{M}}(u)\beta_{-,\mc{N}}(u)-\beta_{-,\mc{M}}(u)\beta_{+,\mc{N}}(u)$;  
\item[($2^*$)] $\ec_{J'}(\mc{M}',\mc{N}')-\ec_J(\mc{M},\mc{N})=\betac_{-,\mc{M}}(u) \betac_{+,\mc{N}}(u)-\betac_{+,\mc{M}}(u)\betac_{-,\mc{N}}(u)$.
\end{enumerate}
In particular, $\e_J(\mc{M},\mc{M})$ and $\ec_J(\mc{M},\mc{M})$ are mutation invariant. So any reachable representation is rigid.
\end{lemma}

\begin{lemma}\label{L:taucommu} {\em \cite[Proposition 7.10]{DF}} The AR-translation $\tau$ commutes with the mutation $\mu_u$ at any vertex $u$.
\end{lemma}

\noindent Finally we mention a long-time conjecture of us.
\begin{conjecture} For a Jacobi-finite QP, any rigid decorated representation can be obtained from a negative representation by a sequence of mutations and some power of $\tau$. 
\end{conjecture}
\noindent This is equivalent to say that $\tau$ acts transitively on the connected components of the cluster complex of $(Q,\S)$ introduced in \cite{DF}.

\section{General Presentations} \label{S:GP}
\subsection{General Presentations} We shall start our discussion by reviewing some results in \cite{DF}. 
We will consider a more general setting where the algebra $A$ is any basic finite-dimensional $k$-algebra, which can be presented as $kQ / I$.

Any $\delta\in \mb{Z}^{Q_0}$ can be written as $\delta = \delta_+ - \delta_-$ where $\delta_+=\max(\delta,0)$ and $\delta_- = \max(-\delta,0)$. Here the maximum is taken coordinate-wise. Lemma \ref{L:homotopy} below motivates the following definition.
$$\PHom_A(\delta):=\Hom_A(P(\delta_-),P(\delta_+)).$$

We say that a {\em general} presentation in $\PHom_A(\delta)$ has property $\heartsuit$ if there is some nonempty open (and thus dense) subset $U$ of $\PHom_A(\delta)$ such that all presentations in $U$ have property $\heartsuit$. For example, a general presentation $d$ in $\PHom_A(\delta)$ has the following properties: $\Hom(d,N)$ has constant dimension for a fixed $N\in \rep A$.
Note that $\E(d,N)$ has constant dimension on $U$ as well. 
We denote these two generic values by $\hom(\delta,N)$ and $\e(\delta,N)$.
Taking $N=A^*$ shows that $\coker(d)$ has a constant dimension vector, which will be denoted by $\dv(\delta)$.

\begin{lemma}[\cite{IOTW}] \label{L:homotopy} For any $\beta_-,\beta_+\in \mb{Z}_{\geq 0}^{Q_0}$, a general presentation in $\Hom_A(P(\beta_-),P(\beta_+))$ is homotopy equivalent to a general presentation in $\PHom_A(\beta_+ - \beta_-)$.
\end{lemma}

\begin{remark} \label{r:genmu} Due to Lemma \ref{L:homotopy} the $\delta$-vector of a general presentation satisfies $\beta_+ = [\delta]_+$ and $\beta_- = [-\delta]_+$.
	In particular, for general presentations, Lemma \ref{L:gdmu}.(1) reduces to the following rule:
	\begin{align}\label{eq:gmu} \delta'(v)= \begin{cases} -\delta(u) & \text{if $v=u$,}\\ 
			\delta(v) + b_{u,v}[-\delta(u)]_+  & \text{if $b_{u,v}<0$,} \\
			\delta(v) + b_{u,v}[\delta(u)]_+ & \text{if $b_{u,v}>0$.} 
		\end{cases}
	\end{align}
	If one likes, one can combine the last two cases into one $\delta'(v)=\delta(v) + [-b_{u,v}]_+\delta(u) + b_{u,v}[\delta(u)]_+$. 
	We call this the mutation rule for $\delta$-vectors. Later when we write $\mu_u(\delta)$, we refer to this rule. 	
\end{remark}

The presentation space $\PHom_A(\delta)$ comes with a natural group action by $$\Aut_A(\delta):=\Aut_A(P(\delta_-))\times \Aut_A(P(\delta_+)).$$
A rigid presentation in $\PHom_A(\delta)$ has a dense $\Aut_A(\delta)$-orbit \cite{DF}.
In particular, a rigid presentation is always general.

As explained in \cite{DF}, the functions $\hom(-,-)$ and $\e(-,-)$ on $\PHom_A(\delta_1)\times\PHom_A(\delta_2)$ are upper semi-continuous.
We will denote their generic values by $\hom(\delta_1,\delta_2)$ and $\e(\delta_1,\delta_2)$.
We say a presentation $d_2$ is a quotient presentation of $d$ if there is a commutative diagram \eqref{eq:subp} with exact rows
\begin{equation}\label{eq:subp} \xymatrix{0\ar[r] & P_-^1\ar[d]_{d_1} \ar[r] &  P_- \ar[d]_d \ar[r] & P_-^2\ar[d]_{d_2} \ar[r] & 0 \\
		0\ar[r] & P_+^1 \ar[r] & P_+ \ar[r] & P_+^2  \ar[r] & 0}
\end{equation}
We may also write this diagram as the exact sequence:
$$0\to d_1\to d\to d_2\to 0.$$

\begin{theorem}[{\cite[Theorem 3.10]{DF}}]\label{T:subp} Let $\beta_+=[\delta_{1}]_+ +[\delta_{2}]_+$ and $\beta_-=[-\delta_{1}]_+ +[-\delta_{2}]_+$. Then a general presentation in $\Hom_A(P(\beta_-),P(\beta_+))$ has a quotient presentation in $\PHom_A(\delta_2)$ if and only if $\e(\delta_1,\delta_2)=0$.
\end{theorem}

The following notion is important throughout this paper.
Let $U$ be a subset of a product $\prod_{i=1}^r X_i$ of topological spaces. We say elements in $U$ can be chosen to be general in some $X_i$ if $p_i(U)$ contains an open subset of $X_i$. 
We say elements in $U$ can be chosen to be {\em general as a pair} in $X_i$ and $X_j$ if $(p_i,p_j)(U)$ contains an open subset of $X_i\times X_j$.
Similarly we can talk about generalness in a triple and so on.

Recall from \cite{DF} that $\E(d_2,d_1)$ is a vector space quotient of $\Hom_A(P_-^2,P_+^1)$ so
any diagram \eqref{eq:subp} gives rises to an element $\eta\in \E(d_2,d_1)$.
\begin{lemma}\label{L:sqgen} Keep the notations in Theorem \ref{T:subp}, and suppose that a general presentation in $\Hom_A(P(\beta_-),P(\beta_+))$ has a quotient presentation $d_2$ in $\PHom_A(\delta_2)$.
Let $d_1$ be the corresponding subpresentation and $\eta$ be the corresponding element in $\E(d_2,d_1)$.
Then we may assume $(d_1,d_2)$ is a general pair and $\eta$ is general in $\E(d_2,d_1)$.
\end{lemma}
\begin{proof} Fix a subrepresentation $P_+^1\cong P([\delta_1]_+)$ of $P_+$ such that $P_+/P_+^1 \cong P([\delta_2]_+)$
	and let $p_+$ be the natural projection $P_+ \to P_+/P_+^1$. We also do the similar thing for $P_-$.
We define the variety 
$$\PHom_A(\delta_1\mid \delta_2):=\{d \in \PHom_A(\delta) \mid p_+ d\mid_{P_-^1}=0 \}.$$
Our assumption that a general presentation in $\PHom_A(\delta)$ has a quotient presentation in $\PHom_A(\delta_2)$ implies that the action morphism
$$\Aut_A(\delta) \times \PHom_A(\delta_1\mid \delta_2) \to \PHom_A(\delta)$$ is dominant. 
It is clear from the definition that 
$$\PHom_A(\delta_1\mid\delta_2) \cong \PHom_A(\delta_1)\times \PHom_A(\delta_2)\times \Hom_A(P_-^2,P_+^1).$$
So any open subset of $\PHom_A(\delta)$ pulls back and projects to a nonempty open subset of $\PHom_A(\delta_1)\times \PHom_A(\delta_2)\times \Hom_A(P_-^2, P_+^1)$.
Hence, we may assume $(d_1,d_2)$ is a general pair.
As $\E(d_2,d_1)$ is a vector space quotient of $\Hom_A(P_-^2,P_+^1)$, the general element in $\Hom_A(P_-^2,P_+^1)$ projects to a general element in $\E(d_2,d_1)$.
\end{proof}

We write $\delta = \delta_1\oplus \delta_2$ if a general presentation of weight $\delta$ is a direct sum of a presentation of weight $\delta_1$ and a presentation of weight $\delta_2$.

\subsection{$\tau$ Permutes Principal Components}
In \cite[Section 2]{DF} we considered the following incidence variety $Z=Z(Y,X)$:
\begin{equation}\label{eq:Z} \{(f,\pi,M)\in Y\times\Hom(P_+, k^\alpha)\times X\mid \pi\in\Hom_A(P_+,M) \text{ and } P_-\xrightarrow{f} P_+\xrightarrow{\pi} M\to 0 \text{ is exact}\}
\end{equation}
for any $\Aut_A(P_-)\times \Aut_A(P_+)$-stable subvariety $Y$ of $\Hom_A(P_-, P_+)$ and $\GL_\alpha$-stable subvariety $X$ of $\rep_\alpha(A)$.
It comes with two projections $p_1: Z\to \Hom_A(P_-,P_+)$ and $p_2:Z\to \rep_\alpha(A)$.
The projection $p_1$ is a principal $\GL_\alpha$-bundle over its image (\cite[Lemma 2.4]{DF}).
In particular, $p_1$ preserves irreducibility.

Let $\alpha$ be the maximal rank of $\Hom_A(P_-, P_+)$,
and $U$ be the open subset of $\Hom_A(P_-, P_+)$ attaining the maximal rank $\alpha$.
As $U$ is irreducible and $p_2p_1^{-1}$ preserves irreducibility, $W=p_2(p_1^{-1}(U))$ lies in a single component of $\rep_\alpha(A)$.
\begin{definition} \label{D:PC}
	We call this component the {\em principal component} of $\delta$, denoted by $\PC(\delta)$.
\end{definition}
\noindent As shown in \cite{P}, principal components are exactly the {\em strongly reduced components} introduced in \cite{GLS}. This fact also follows from our Lemma \ref{L:gqiso} below. 

\begin{theorem}[{\cite[Theorem 2.3]{DF} }]
\label{L:P2R} For the above dimension vector $\alpha$, the image of the projection 
$p_2: Z(\Hom_A(P_-,P_+), \rep_\alpha(A)) \to \rep_\alpha(A)$
is open in $\rep_\alpha(A)$.
\end{theorem}

\begin{lemma} \label{L:U0} There is a nonempty open $\Aut_A(\delta)$-stable subset $U_0$ of $\PHom_A(\delta)$ such that there is a morphism 
	$\pi: U_0 \to \PC(\delta)$ such that $\pi(d)$ is isomorphic to $\coker(d)$ for all $d\in U_0$.	
\end{lemma}
\begin{proof} Since the map $p_1$ is a principal $\GL_\alpha$-bundle, it has a section $\iota_1$ on some open subset of $U_0$ of $\PHom_A(\delta)$.
The composition $p_2\iota_1$ is our desired morphism.
\end{proof}

By a $G$-variety, we mean a variety with an action of an algebraic group $G$. 
By a theorem of Rosenlicht (\cite{R}, see also \cite[Theorem 6.2]{Do}), any irreducible $G$-variety $X$ contains a non-empty open $G$-stable subset $X_0$ which admits a {\em geometric quotient} $X_0/G$. For the definition of geometric quotient, we refer readers to \cite[6.1]{Do}.

We shall denote by $\e_A(\delta)$ the generic value of $\e_A(d,d)$ for $d\in\PHom_A(\delta)$.
Note the difference between $\e_A(\delta)$ and $\e_A(\delta,\delta)$.
The latter does not require the two arguments $d_1$ and $d_2$ to be the same in $\e_A(d_1,d_2)$.
\begin{lemma} \label{L:gqiso} There is a non-empty open $\Aut_A(\delta)$-stable subset $U$ of $\PHom_A(\delta)$ and a non-empty open $\GL_{\dv(\delta)}$-stable subset $W$ of its principal component $\PC(\delta)$ such that there is an isomorphism of geometric quotients
	$$U / \Aut_A(\delta) \to W / \GL_{\dv(\delta)}.$$
Moreover the dimension of the quotient $U / \Aut_A(\delta)$ is equal to $\e_A(\delta)$.
\end{lemma}
\begin{proof} Let $U_0$ be an open subset of $\PHom_A(\delta)$ as in Lemma \ref{L:U0}. We get a morphism of varieties $\pi: U_0 \to \PC(\delta)$.
Let $W$ be the open subset of $\PC(\delta)$ such that $q: W\to W/\GL_{\dv(\delta)}$ is a geometric quotient.
Let $U_1$ be the open subset of $\PHom_A(\delta)$ such that $U_1\to U_1/\Aut_A(\delta)$ is a geometric quotient, and set $U = U_0\cap U_1\cap \pi^{-1}(W)$.
The composition gives $q\pi\mid_U: U\twoheadrightarrow W/\GL_{\dv(\delta)}$.
Recall that a geometric quotient is also a categorical quotient \cite{Do}.
This map is constant on the $\Aut_A(\delta)$-orbits of $U$ so it descends to $U/\Aut_A(\delta) \to W/\GL_{\dv(\delta)}$.
This is a bijective regular map, and hence an isomorphism on some open subsets, which we still denote by $U/\Aut_A(\delta)$ and $W/\GL_{\dv(\delta)}$.

For the statement about the dimension, we have that
\begin{align}\notag \dim (U / \Aut_A(\delta)) &= \dim U - \dim \Aut_A(\delta) + \min_{d\in\PHom_A(\delta)} (\dim\Aut_A(\delta)_d) \\
\label{eq:dim} &= \dim \PHom_A(\delta) - \dim \Aut_A(\delta) + \dim (\Aut_A(\delta)_d) 
\end{align}
where $d$ is a general presentation of weight $\delta$ and $\Aut_A(\delta)_d$ is the stabilizer of $d$ in $\Aut_A(\delta)$.
It follows from \cite[Lemma 3.6]{DF} that the Lie algebra of $\Aut_A(\delta)_d$ can be identified with the kernel of the map 
$f=(d_+, -d_-)$, where $d_+, d_-$ is the induced map by applying $\Hom_A(-,P_+)$ and $\Hom_A(P_-,-)$ to $d$ respectively.
\begin{equation} \label{eq:doublec} \xymatrix{\Hom_A(P_+,P_-) \ar[r] \ar[d] &\Hom_A(P_-, P_-) \ar[d]^{-d_-} \\  \Hom_A(P_+,P_+) \ar[r]^{d_+} &\Hom_A(P_-,P_+)} \end{equation}
Consider the complex induced from \eqref{eq:doublec}:
$$0 \to \ker(f) \to \Hom_A(P_-,P_-) \oplus \Hom_A(P_+,P_+) \xrightarrow{f} \Hom_A(P_-,P_+) \to \E_A(d,d) \to 0.$$
We conclude that $\eqref{eq:dim} = \e_A(\delta)$ as desired.
\end{proof}
\begin{remark} \label{r:gqiso} Slightly modifying this proof we can easily show that for any $\Aut_A(\delta)$-stable subset $X$ of $\PHom_A(\delta)$ which maps onto an open subset of some irreducible component $C$ of $\rep_\alpha(A)$, there exist open subsets $U$ and $W$ of $X$ and $C$ respectively such that there is an isomorphism of geometric quotients $U/\Aut_A(\delta) \to W/\GL_\alpha$. 
	
We also remark that there is a similar statement for $\IHom(\dtc)$. In this case, the dimension formula should read as
$\dim (\check{U} / \Aut_A(\dtc)) = \ec_A(\dtc)$.
\end{remark}

From now on, we shall assume the algebra $A$ is a Jacobian algebra $J$ as before. 
A decorated representation $\mc{M}$ is called general of weight $\delta$ if it corresponds to a general presentation of weight $\delta$. In this case, $M$ is a general representation in $\PC(\delta)$.
By abuse of language, $M$ is also called a general representation of weight $\delta$.
By \eqref{eq:delta2dual} a general representation of weight $\delta$ has $\dtc$-vector $\delta+\dv(\delta)B_Q$, which will be denoted by $\dtc$.
By the remarks after Definition \ref{D:HomE}, the $\hom(\delta,\ep)$ and $\e(\delta,\ep)$ defined in terms of presentations is nothing but the generic (minimal) values of $\hom_J(M,N)$ and $\e_J(M,N)+ \delta^- (\dv(\ep))$ on $\PC(\delta)\times \PC(\ep)$. Here, $\delta^-$ is the nonpositive vector that appears in the decomposition $\delta=\delta'\oplus \delta^-$ with $\delta'$ {\em negative-free}, which means that a general presentation in $\PHom_J(\delta')$ does not have a direct summand of form $P\to 0$.
\begin{lemma} \label{L:eceq} For $\dtc=\delta+\dv(\delta)B_Q$, we have that $\e_J(\delta) = \ec_J(\dtc)$.
\end{lemma}
\begin{proof}
Let $d$ be a presentation of weight $\delta$ such that $\e_J(d,d)=\e_J(\delta)$. 
Let $\dc$ be the corresponding injective presentation, that is, $d=d_{\mc{M}}$ and $\dc=\dc_{\mc{M}}$.
Then $\dc$ has weight $\dtc$ by \eqref{eq:delta2dual}.
We have that $\ec_J(\dtc) \leq \ec_J(\dc,\dc) = \e_J(d,d) = \e_J(\delta)$ by Lemma \ref{L:H2E}.
Similarly we can get the other inequality.
\end{proof}

Due to the relation $\delta_{\tau\mc{M}} = -\dtc_{\mc{M}}$ and \eqref{eq:delta2dual},
we have that for a general presentation $d$ of weight $\delta$, the $\delta$-vector of $\tau d$ is constant.
We denote this constant vector by $\tau\delta$.
\begin{theorem}\label{T:genpi} {\ }\begin{enumerate}
\item	$\PHom_J(\delta)$ and $\IHom_J(\dtc)$ have the same principal component where $\dtc=\delta + \dv(\delta) B_Q$.
\item   $M$ is a general representation in $\PC(\delta)$ if and only if $\tauh \mc{M}$ is a general representation in $\PC(\tau\delta)$.
\end{enumerate}
\end{theorem}
\begin{proof} (1). It is enough to show that there is an open subset $U$ of $\PHom_J(\delta)$ and an open subset $\check{U}$ of $\IHom_J(\dtc)$ such that they correspond to the same open subset $W$ in the principal component $\PC(\delta)$.
	
Let $U$ and $W$ be the open subset of $\PHom_J(\delta)$ and $\PC(\delta)$ as in Lemma \ref{L:gqiso}.
Let $\check{Z}(Y,X)$ be the analogous variety of \eqref{eq:Z} for injective presentations with projections $p_1$ and $p_2$.
Let $\check{U}'=p_1(p_2^{-1}(W))$ be the (constructible) subset of $\IHom_J(\dtc)$.
By possibly shrinking $\check{U}'$ and $W$ we may assume the isomorphism (see Remark \ref{r:gqiso})
$$\check{U}' / \Aut_A(\dtc) \cong W / \GL_{\dv(\delta)}.$$
So by Lemma \ref{L:gqiso} we have that
\begin{align*} \dim (\check{U}' / \Aut_A(\dtc)) = \dim(W / \GL_{\dv(\delta)}) = \dim(U/ \Aut_A(\delta)) = \e_J(\delta),
\end{align*}
which is equal to $\ec_J(\dtc)$ by Lemma \ref{L:eceq}.
But $\dim \check{U} / \Aut_A(\dtc) = \ec_J(\dtc)$ as well, where $\check{U}$ is an open subset of $\IHom_J(\dtc)$ as claimed in Remark \ref{r:gqiso}.
So $\dim \check{U}' = \dim \check{U}$. Hence, $\check{U}'$ is in fact open in $\IHom_J(\dtc)$.

(2). Let $\iota$ be the isomorphism $U/\Aut_A(\delta) \to W/\GL_{\dv(\delta)}$ as in Lemma \ref{L:gqiso}.
By the part (1), $\check{U}/\Aut_A(-\delta)$ is isomorphic to $\check{W}/\GL_{\dv(\tau\delta)}$ for some open subset $\check{W}$ in $\PC(\tau\delta)$.
Recall the Nakayama functor $\nu$. It has an obvious algebraic lifting $\PHom(\delta) \to \IHom(-\delta)$ for each $\delta$, which is an isomorphism and sends a linear combination of paths $\sum_{i} c_ip_i$ to $\sum_{i} c_ip_i^*$. 
By properly adjusting the open sets $U$ and $\check{U}$, this isomorphism descends to the geometric quotients $U/\Aut_A(\delta) \xrightarrow{\nu} \check{U}/\Aut_A(-\delta)$.
We thus obtain the following diagram (if necessary we may shrink $W$ and $\check{W}$)
$$\xymatrix{U/\Aut_A(\delta) \ar[r]^{\iota}_{\cong} \ar[d]_{\nu} & W/\GL_{\dv(\delta)} \ar[d]^{\check{\iota} \nu \iota^{-1}} \\
	\check{U}/\Aut_A(-\delta) \ar[r]^{\check{\iota}}_{\cong} & \check{W}/\GL_{\dv(\tau\delta)}
}$$
As $W$ and $\check{W}$ are open in $\PC(\delta)$ and $\PC(\tau\delta)$,
the isomorphism $W/\GL_{\dv(\delta)}\cong \check{W}/\GL_{\dv(\tau\delta)}$ implies (2).
\end{proof}

In the end, we point out that mutations also act on principal components.
\begin{lemma}[{\cite[Theorem 1.10]{GLFS}}] \label{L:genmu} There is an open subset $U$ in $\PC(\delta)$ and an open subset $U'$ in $\PC(\mu_u(\delta))$ such that the orbit of $\mu_u(M)$ lies in $U'$ if and only if the orbit of $M$ lies in $U$.
In particular, $M$ is a general representation in $\PC(\delta)$ if and only if $\mu_u(M)$ forgetting decoration is a general representation in $\PC(\mu_u(\delta))$.
\end{lemma}

\section{A Review on $2$-Calabi-Yau Triangulated Categories} \label{S:2CY}
\subsection{The Cluster Category $\CQ$}
C. Amiot introduced in \cite{Am} a triangulated category $\CQ$ associated to a quiver with potential $(Q,\mc{S})$.
Let $\Gamma=\Gamma_{Q,\mc{S}}$ be the complete Ginzburg's dg-algebra attached to $(Q,\mc{S})$ \cite{G}, and $\mc{D}\Gamma$ be its derived category.
The perfect derived category $\op{per}\Gamma$ of $\Gamma$ is the smallest full triangulated subcategory of $\mc{D}\Gamma$ containing $\Gamma$ and closed under taking direct summands. 
Denote by $\mc{D}_{fd}\Gamma$ the full subcategory of $\mc{D}\Gamma$ whose objects are those of $\mc{D}\Gamma$ with finite-dimensional total homology. As shown in \cite[Theorem
2.17]{KY}, the category $\mc{D}_{fd}\Gamma$ is a triangulated subcategory of $\op{per}\Gamma$.
The {\em cluster category} $\CQ$ of $(Q,\S)$ is defined as the idempotent completion of the triangulated
quotient $(\op{per}\Gamma)/\mc{D}_{fd}\Gamma$.

When $(Q,\S)$ is Jacobi-finite, the category $\CQ$ is $\Hom$-finite and $2$-Calabi–Yau, and admits a basic {\em cluster-tilting} object $\T=\Sigma^{-1}\Gamma$.
Its endomorphism algebra is isomorphic to the Jacobian algebra $J(Q,\mc{S})$.
So $\T$ decomposes as $\T=\bigoplus_{u\in Q_0} \T_u$.
Recall that a triangulated category $\C$ is {\em $2$-Calabi-Yau} if there is a bifunctorial isomorphism
$$\C(\b{L},\Sigma \b{N})\cong \C(\b{N},\Sigma \b{L})^*.$$
A cluster-tilting object is by definition an object $\T$ of $\C$ satisfying
\begin{enumerate} \item $\C(\T,\Sigma\T)=0$ and
	\item for any $\b{M}$ in $\C$, if $\C(\b{M},\Sigma\T)=0$, then $\b{M}$ belongs to the full additive subcategory $\add\T$.
\end{enumerate}

We keep $k$ as an algebraically closed field of characteristic zero.
Throughout we will write $\mc{C}$ for $\C_{Q,\mc{S}}$ though some of the definitions and results hold in any $\Hom$-finite, $2$-Calabi–Yau, Krull–Schmidt $k$-category which admits a basic cluster-tilting object $\T$. Let $J$ be the endomorphism algebra of $\T$ in $\C$, and denote by $\mod J$ the category of finite-dimensional right $J$-modules.
As shown in \cite{KR} the functor $F:\C\to \mod J$ sending $\b{M}$ to $\C(\T,\b{M})$, induces an equivalence of categories:
\begin{equation}\label{eq:equiv} \C/(\Sigma \T) \cong \mod J, \end{equation}
where $(\T)$ denotes the ideal of morphisms of $\C$ which factor through an object in $\add \T$.
This equivalence restricts to the full subcategories: $\add \T \to \proj J$, which allows us lift $J$-modules to $\mc{C}$ using projective presentations.

This equivalence can be slightly extended to incorporate the decorated representations (see \cite{P0} for more details).
Let $\b{M}$ be an object in $\C$ of the form $\b{M} = \b{M}' \oplus  \bigoplus_{u\in Q_0} m_u \Sigma\T_u$ where $\b{M}'$ has no direct summands in $\add \Sigma\T$.
Such an $\b{M}$ will correspond to the decorated representation $(F \b{M}',\ \bigoplus_{u\in Q_0}m_u S_u)$.
We denote this map by $\wtd{F}=\wtd{F}_{Q,\mc{S}}$, which is denoted by $\Phi$ in \cite{P0}.
If no potential confusion is possible, throughout we will write $\b{M}$ for a lift of $M$ or $\mc{M}$, and denote ${F}\b{M}$ and $\wtd{F}\b{M}$ by $M$ and $\mc{M}$ respectively.

We also have the following analogue of the space $\E$ in $\C$. Following \cite{P}, for $\b{d}'\in \C(\T_-',\T_+')$ and $\b{d}''\in \C(\T_-'',\T_+'')$ 
we define $\E(\b{d}',\b{d}'')$ to be the vector space $\Hom_{K^b(\add\T)}(\b{d}', \Sigma\b{d}'')$. We denote $F\b{d}'$ and $F\b{d}''$ by $d'$ and $d''$.
\begin{lemma}[{\cite[Proposition 3.10]{P}}] \label{L:EinC} The space $\E(\b{d}',\b{d}'')$ is isomorphic to $\E(d', d'')$, and it can be identified with $(\Sigma\T)(\b{N},\Sigma\b{L})$ where $\b{N}$ and $\b{L}$ are the cones of $\b{d}'$ and $\b{d}''$ respectively.
	Dually $\Ec(\b{\dc}',\b{\dc}'')$ is isomorphic to $\Ec(\dc', \dc'')$, and it can be identified with $(\Sigma\T)(\Sigma^{-1}\b{N},\b{L})$ where $\b{N}$ and $\b{L}$ are the fibres of $\b{d}'$ and $\b{d}''$ respectively.
\end{lemma}

\begin{lemma}[{\cite[Lemma 3.3]{Pa}}] \label{L:EC} We have bifunctorial isomorphisms
	$$\C/_{(\Sigma\T)}(\Sigma\b{N}, \b{L}) \cong (\Sigma\T)(\b{L}, \Sigma\b{N})^*.$$
\end{lemma}

\begin{corollary}\label{C:e0} Let $\b{L}\to \b{M}\to\b{N}\to\Sigma\b{L}$ be a triangle in $\C$.
	\begin{enumerate}
		\item  If $\e(\mc{L},\mc{N})=0$, then every $f\in \C(\b{N},\Sigma\b{L})$ factors through $\Sigma \T$.
		\item  If $\ec(\mc{L},\mc{N})=0$, then every $f\in \C(\Sigma^{-1}\b{N},\b{L})$ factors through $\Sigma \T$.
		\item If $\e(\mc{L},\mc{N})=\e(\mc{N},\mc{L})=0$, then $\mc{C}(\b{L},\Sigma \b{N})=\mc{C}(\b{N},\Sigma \b{L})=0$.
	\end{enumerate}
\end{corollary}
\begin{proof} (1). Since $\e(\mc{L},\mc{N})=0$, we have that $(\Sigma \T)(\b{L},\Sigma \b{N})=0$. So by Lemma \ref{L:EC} every $f\in \C(\b{N},\Sigma\b{L})$ factor through $\Sigma \T$. The proof of (2) is similar.
(3). If $\e(\mc{L},\mc{N})=\e(\mc{N},\mc{L})=0$, then $(\Sigma \T)(\b{L},\Sigma \b{N})=(\Sigma \T)(\b{N},\Sigma \b{L})=0$. By Lemma \ref{L:EC} $\mc{C}(\b{L},\Sigma \b{N})=\mc{C}(\b{N},\Sigma \b{L})=0$.
\end{proof}

\subsection{Palu's Formula on Indices}
Let $\b{M}$ be an object of $\C$. There exist triangles (see \cite{KR}) 
$$\T_-\to \T_+\to \b{M} \to \Sigma \T_- \ \text{ and }\ \b{M}\to \Sigma^2 \T^+ \to \Sigma^2 \T^-\to \Sigma \b{M}$$ 
with $\T_+,\T_-, \T^+,\T^-\in \op{add} \T$.
\begin{definition} The {\em index} and {\em coindex} of $\b{M}$ with respect to $\T$ are the classes in $K_0(\add \T)$:
	$$\ind_{\T}(\b{M}) = [\T_+] - [\T_-]\ \text{ and }\ \op{coind}_{\T}(\b{M}) = [\T^+]-[\T^-].$$
\end{definition}
\noindent Similar to the projective representations, we denote $\T(\beta) = \bigoplus_{u\in Q_0}\beta(u) \T_u$ where $\T_u$ is the indecomposable direct summands of $\T$ corresponding to $u$. 
In this way, the index and coindex can be naturally identified with a vector in $\mb{Z}^{Q_0}$: if $\T_+ = \T(\beta_+)$ and $\T_- = \T(\beta_-)$, then $\ind_{\T}(\b{M}) = \beta_+ - \beta_-$. Under this identification, the following lemma is obvious from the equivalence \eqref{eq:equiv} and the fact that $\ind_{\T}(\Sigma\T_u) = e_u$ where $e_u$ is the standard basis vector in $\mb{Z}^{Q_0}$ corresponding to the vertex $u$.
\begin{lemma}\label{L:indwt} As a vector in $\mb{Z}^{Q_0}$, the index of $\b{M}$ is equal to the weight of $\wtd{F}(\b{M})$.
\end{lemma}
\noindent We also define the presentation space $\T\!\Hom(\delta):=\C(\T([\delta]_+),\T([-\delta]_+))$.
By a {\em general} element $\b{M}$ of index $\delta$, we mean that $\b{M}$ is the cone of a general element in $\T\!\Hom(\delta)$.

Here is Palu's formula on indices in a triangle.
\begin{lemma}[{\cite[Proposition 2.2]{Pa}}] \label{L:index} Let $\b{L}\xrightarrow{\b{f}} \b{M}\xrightarrow{\b{g}} \b{N}\to \Sigma\b{L}$ be a triangle in $\C$. Take $\b{C}\in\C$ (resp. $\b{K}\in\C$) to be any lift of $\coker F\b{g}$ (resp. $\ker F\b{f}$). Then
	\begin{align*} \ind \b{M} &= \ind \b{L} + \ind \b{N} - \ind \b{C} - \ind \Sigma^{-1}\b{C}; \\
		\op{coind} \b{M} &= \op{coind}\b{L} + \op{coind}\b{N} - \op{coind}\b{K} - \op{coind} \Sigma \b{K}.
	\end{align*}
\end{lemma}

\begin{corollary}\label{C:long} A triangle $\b{L}\to \b{M}\to \b{N}\to \Sigma \b{L}$ in $\C$ gives a long exact sequence in $\mod J$:
\begin{equation} \label{eq:long} \cdots \to \tauh^{-1} \mc{L}\to \tauh^{-1} \mc{M}\to \tauh^{-1} \mc{N} \xrightarrow{\theta'} L\to M\to N \xrightarrow{\theta} \tauh \mc{L} \to \tauh \mc{M} \to \tauh \mc{N} \to \tauh^2 \mc{L} \to \cdots.
\end{equation}
with $\delta_{\mc{M}} = \delta_{\mc{L}} + \delta_{\mc{N}} + \rank(\theta) B$ and $\dtc_{\mc{M}} = \dtc_{\mc{L}} + \dtc_{\mc{N}} - \rank(\theta') B$.	
\end{corollary}
\begin{proof} We apply the functor $F=\C(\T,-)$ to the above triangle, and note that $\C(\T,\Sigma^i \b{M}) \cong \tauh^i \mc{M}$.
So we get the desired long exact sequence. Since $\delta_{\mc{M}} = \ind\b{M}$ by Lemma \ref{L:indwt}, the last formula then follows from Lemma \ref{L:index}.
\end{proof}
\noindent We had a direct constructive proof of a similar but weaker corollary without referring to the category $\C$. We'd like to thank 
Bernhard Keller for pointing out the connection to the 2-Calabi-Yau categories. This saved us at least two pages of space.

\subsection{Lifting Mutations}
In \cite{KY} Keller and Yang lifted Derksen-Weyman-Zelevinsky's mutation to the category $\mc{D}\Gamma$.
The lifted mutation $\wtd{\mu}_u^-$ is in fact a triangle equivalence $\mc{D}\Gamma \to \mc{D}\Gamma'$, where $\Gamma'$ is the complete Ginzburg dg algebra of $\wtd{\mu}_u(Q,\mc{S})$.
This equivalence restricts to the subcategories $\op{per}\Gamma\to \op{per}\Gamma'$ and $\mc{D}_{fd}\Gamma \to \mc{D}_{fd}\Gamma'$.
There are similar statements for the reduced quivers with potentials (see \cite[Theorem 3.2]{KY}).
In particular, it induces a triangle equivalence $\CQ \to \C_{\mu_u(Q,\mc{S})}$, denoted by $\br{\mu}_u^-$.
The lifted mutation is compatible with the ordinary one in the following sense.
\begin{theorem}[{\em cf.} {\cite[Proposition 4.1]{P0}}] \label{T:twomu} If $u \in Q_0$ is not on any cycle of length $2$, then for any object $\b{M}$ of $\mc{C}_{Q,\S}$, we have that 
\begin{equation}\label{eq:muwtd} \wtd{F}_{{\mu}_u(Q,\mc{S})}(\br{\mu}_u^-(\b{M})) = {\mu}_u(\wtd{F}_{Q,\S}(\b{M})).
\end{equation}
\end{theorem}
\begin{proof} This is a simple variation of \cite[Proposition 4.1]{P0}, in which we already have that
\begin{equation}\label{eq:mubar} \wtd{F}_{\wtd{\mu}_u(Q,\mc{S})}(\wtd{\mu}_u^-(\b{M})) = \wtd{\mu}_u(\wtd{F}_{Q,\S}(\b{M})).
\end{equation}
It follows from \cite[Lemma 2.10]{KY} that there is a quasi-isomorphism $\Gamma \to \Gamma_{\op{red}}$,
where $\Gamma_{\op{red}}$ is the complete Ginzburg dg algebra associated to the reduced quiver with potential $(Q_{\red}, \S_{\red})$.
Since quasi-isomorphisms of dg algebras induce triangle equivalences in their derived categories,
we thus get a triangle equivalence $\red: \mc{C}_{Q,\S} \to \mc{C}_{Q_{\red}, \S_{\red}}$.
Comparing \eqref{eq:muwtd} with \eqref{eq:mubar}, we remain to show that
$\wtd{F}_{{\mu}_u(Q,\mc{S})}\circ \red = \wtd{F}_{\wtd{\mu}_u(Q,\S)}$.
It suffices to show that
$$ \wtd{F}_{Q_{\red},\mc{S}_{\red}}\circ \red = \wtd{F}_{Q,\S},\ \text{
 or equivalently, }\ 
\mc{C}_{\red}(\Sigma^{-1}\Gamma_{\red},\ \red(-)) = \mc{C}(\Sigma^{-1}\Gamma,\ -).$$
But this also follows from the triangle equivalence $\red: \mc{C}_{Q,\S} \to \mc{C}_{Q_{\red}, \S_{\red}}$.
\end{proof}

Given an admissible sequence of mutations $\mub=\mu_{u_r}\cdots\mu_{u_2}\mu_{u_1}$, we have a sequence of triangle equivalences
\begin{equation} \label{eq:trieq} \br{\mu}_{\b{u}}^-: \CQ \to \C_{\mu_{u_1}(Q,\mc{S})}\to \C_{\mu_{u_2}\mu_{u_1}(Q,\mc{S})} \to \cdots\to \C_{\mub(Q,\mc{S})}. \end{equation}
Let the functor $\wtd{F}'$ be the composition $\wtd{F}_{\mub(Q,\mc{S})}\circ \br{\mu}_{\b{u}}^-$. We write $J'$ for the Jacobian algebra of $\mub(Q,\mc{S})$.
\begin{lemma} \label{L:gendes} Given an admissible sequence of mutations $\mub=\mu_{u_r}\cdots\mu_{u_2}\mu_{u_1}$, the functor $\wtd{F}'$ sends a general object $\b{M}$ of index $\delta$ in $\C$ to a general object of weight $\mub(\delta)$ in $\mc{R}ep J'$,
	and a general morphism $\b{g}$ in $\C(\b{M},\b{N})$	to a general homomorphism in $\Hom_{J'}(\wtd{F}'\b{M}, \wtd{F}'\b{N})$.
\end{lemma}
\begin{proof} Due to Theorem \ref{T:twomu}, $\wtd{F}' = \mub\circ \wtd{F}_{Q,\mc{S}}$.
	Note that $\wtd{F}_{Q,\mc{S}}(\b{M})$ is general of weight $\delta$.
	So it suffices to show that a general representation of weight $\delta$ will be mutated to a general representation of weight $\mub(\delta)$. But this follows from Lemma \ref{L:genmu}.
	Since $\mub^-$ is a triangle equivalence, $\mub^-(\b{g})$ is general in $\mc{C}(\b{M}', \b{N}')$ where $\b{M}'=\mub^-(\b{M})$ and $\b{N}'=\mub^-(\b{N})$.
	For $\T'=\Sigma^{-1}\Gamma_{\mu_\b{u}(Q,\mc{S})}$, $\C(\b{M}',\b{N}') \to \C(\b{M}',\b{N}')/(\Sigma \T')(\b{M}',\b{N}') \cong \Hom_{J'}(M', N')$ is a vector space projection. 
	In particular, it is an open map so a general morphism in $\C(\b{M}',\b{N}')$ descends to a general element in $\Hom_{J'}(M',N')$.
\end{proof}

\subsection{The Lowering and Raising Operators in $\C$}
\begin{definition}\label{D:rlc} For any object $\b{E}\in\C$, we define the operator $\b{r}_\b{E}$ on $K_0(\add \T)$ as follows.
Pick an object $\b{M}$ of index $\delta$, and let $\b{g}$ be a morphism $\C(\b{E},\Sigma\b{M})$. We complete $\b{g}$ to a triangle 
\begin{equation}\label{eq:rtri} \Sigma^{-1}\b{E}\to \b{M}\to \b{R} \to \b{E}. \end{equation} 
By Lemma \ref{L:index}	the index of $\b{R}$ is constant if $\b{M}$ is general and $\b{g}$ is general in $\C(\b{E},\Sigma\b{M})$.
	We define $\b{r}_\b{E}(\delta)$ to be such a constant index of $\b{R}$.
	Similarly we can take $\b{g}$ to be a general morphism $\C(\b{M},\b{E})$ and complete it to 
	\begin{equation}\label{eq:ltri} \Sigma^{-1}\b{E}\to \b{L}\to \b{M}\to \b{E}, \end{equation}
and define $\b{l}_\b{E}(\delta)$ to be the index of $\b{L}$. 	
	If working with the coindices, we get another version of $\b{r}_\b{E}$ and $\b{l}_\b{E}$ which will be denoted by $\b{r}^\b{E}$ and $\b{l}^\b{E}$.	
If $\b{E}$ is a general object of index $\ep$, then $\b{r}_{\b{E}}$ is constant by Lemma \ref{L:index} and Lemma \ref{L:genrank} later.
We will denote $\b{r}_\b{E}$ and $\b{l}_\b{E}$ by $\b{r}_\ep$ and $\b{l}_\ep$.
\end{definition}
\noindent It is unclear from the definition if the above $\b{R}$ or $\b{L}$ can be chosen as a general element of index $\b{r}_\ep$ or $\b{l}_\ep$.
Later we will see that this indeed can be done in most cases (conjecturally for any nondegenerate $(Q,\mc{S})$).
We remark that the general morphism $\b{g}$ in the above definition can be chosen from other periods of the triangle.

\begin{definition} \label{D:liftc} We say $\b{r}_\ep$ (resp. $\b{l}_\ep$) can be generically lifted to a triangle at $\delta$ if the above $\b{R}$ (resp. $\b{L}$) can be chosen as a general element of index $\b{r}_\ep(\delta)$ (resp. $\b{l}_\ep(\delta)$).
To be more precise, this means that a general object $\b{R}$ (resp. $\b{L}$) of index $\b{r}_\ep(\delta)$ (resp. $\b{l}_\ep(\delta)$) fits into the triangle \eqref{eq:rtri} (resp. \eqref{eq:ltri}) such that $(\b{M},\b{E})$ is a general pair of index $(\delta,\ep)$ and $\b{g}$ is a general morphism in $\C(\b{E},\Sigma\b{M})$ (resp. $\C(\b{M},\b{E})$).
\end{definition}
\noindent We emphasis that in the above definition we do not ask any pair other than $(\b{M},\b{E})$ to be general. For example, $(\b{M},\b{R})$ or $(\b{R},\b{E})$ may not be general as a pair.


\begin{lemma} \label{L:rle=0} If $\C(\Sigma^{-1}\b{E}_2, \b{E}_1)=0$, then 	we have that $\b{r}_{\b{E}_1\oplus \b{E}_2} = \b{r}_{\b{E}_2}\b{r}_{\b{E}_1}$ and $\b{l}_{\b{E}_1\oplus \b{E}_2} = \b{l}_{\b{E}_1}\b{l}_{\b{E}_2}$.
If in addition $\C(\Sigma^{-1}\b{E}_1, \b{E}_2)=0$ then $\b{r}_{\b{E}_1}$ and $\b{r}_{\b{E}_2}$ commute and $\b{l}_{\b{E}_1}$ and $\b{l}_{\b{E}_2}$ commute.
\end{lemma}
\begin{proof} We only prove the statement for $\b{r}_{\b{E}}$ because the argument for $\b{l}_{\b{E}}$ is similar.
Let $(\b{f}_1,\b{f}_2)$ be a general morphism in $\C(\Sigma^{-1}(\b{E}_1\oplus\b{E}_2),\b{M})$. 
Then $\b{f}_1$ and $\b{f}_2$ are general in $\C(\Sigma^{-1}\b{E}_1,\b{M})$ and $\C(\Sigma^{-1}\b{E}_2,\b{M})$.
Apply $\C(\Sigma^{-1}\b{E}_2,-)$ to the triangle in the second row of \eqref{eq:nine-diagram}, we get 
$$\C(\Sigma^{-1}\b{E}_2, \b{M}) \to \C(\Sigma^{-1}\b{E}_2, \b{R}_1) \to \C(\Sigma^{-1}\b{E}_2, \b{E}_1) =0.$$
So we may choose a morphism $\b{g}$ general in $\C(\Sigma^{-1}\b{E}_2, \b{R}_1)$ making the upper right square commute.
Since the triangle of first row is split (in particular contractible), we can complete the first two rows to a nine-diagram below \cite{N}.
By construction and definition, the index of $\b{R}_{1\oplus 2}$ is $\b{r}_{\b{E}_1\oplus \b{E}_2}(\delta)$ while 
the index of $\b{R}_{1,2}$ is $\b{r}_{\b{E}_2}\b{r}_{\b{E}_1}(\delta)$.
But these two objects are isomorphic.	
\begin{equation}\label{eq:nine-diagram} \xymatrix{
		\Sigma^{-1}\b{E}_1 \ar[r]\ar@{=}[d]_{} & \Sigma^{-1}(\b{E}_1\oplus\b{E}_2) \ar[r]\ar[d]^{(\b{f}_1,\b{f}_2)} & \Sigma^{-1}\b{E}_2 \ar[d]^{\b{g}} \ar[r] & \b{E}_1\\ 
		\Sigma^{-1}\b{E}_1 \ar[r]^{\b{f}_1}\ar[d] & \b{M} \ar[r]\ar[d] & \b{R}_1 \ar@{-->}[d]\ar[r] & \b{E}_1\\ 
		0 \ar[r]\ar[d] & \b{R}_{1\oplus 2} \ar@<.5ex>@{--}[r] \ar@<-.5ex>@{--}[r] \ar[d] & \b{R}_{1,2} \ar[d]\ar[r]& 0\\
		\b{E}_1 & \b{E}_1 \oplus \b{E}_2 &  \b{E}_2 &
	}
\end{equation}
If in addition $\C(\Sigma^{-1}\b{E}_1, \b{E}_2)=0$, then $\b{r}_{\b{E}_1\oplus \b{E}_2} = \b{r}_{\b{E}_1}\b{r}_{\b{E}_2}$ so $\b{r}_{\b{E}_1}$ and $\b{r}_{\b{E}_2}$ commute.

\end{proof}


\section{The General Ranks} \label{S:GR}
\subsection{The Raising and Lowering Operators} 
Schofield introduced the general rank for quiver representations in his theory of general representations \cite{S}.
The following lemma is a straightforward generalization of \cite[Lemma 5.1]{S}.
\begin{lemma} \label{L:genrank} Let $A$ be a finite-dimensional algebra. Given any two irreducible closed sets $X$ and $Y$ in representation varieties of $A$, there is an open subset $U$ of $X \times Y$ and a dimension vector $\gamma$ such that for $(M,N)\in U$ we have that $\hom_A(M,N)$ is minimal and $\{\phi\in\Hom_A(M,N)\mid \rank \phi = \gamma \}$ is open and non-empty in $\Hom_A(M,N)$. 
\end{lemma}
\noindent Below the algebra $A$ will always be the Jacobian algebra of some quiver with potential.

\begin{definition} If one of $X$ and $Y$ is a single representation, say $Y=\{E\}$, and $X$ is the principal component $\PC(\delta)$, then the above dimension vector is denoted by $\rank(\delta,E)$.
If $X=\PC(\delta)$ and $Y=\PC(\ep)$, then $\gamma$ is called the {\em general rank} from $\delta$ to $\ep$, denoted by $\rank(\delta,\ep)$. There are obvious variations if we replace $\delta$ or $\ep$ by a $\dtc$-vector.
\end{definition}

\begin{example} If $(\delta,\ep)$ is an {\em exchange pair}, that is, $\delta$ and $\ep$ are indecomposable rigid with $\e(\delta,\ep)=1$ and $\e(\ep,\delta)=0$, then $\rank(\ep, \tau\delta)\neq 0$ is the {\em $c$-vector} associated to the exchange pair \cite[Section 7]{Ft}.
For the connection to the corresponding concept in cluster algebras, we refer readers to \cite[Remark 7.7]{Ft}.
\end{example}

\noindent For quivers with potentials, we have that $\PC(\delta)=\PC(\dtc)$ by Theorem \ref{T:genpi}, so $\rank(\delta,\ep) = \rank(\dtc,\ep)$. In what follows, whenever the quiver $Q$ is fixed, we will write $B$ for $B_Q$.

\begin{definition} \label{D:rlgrank} For any decorated representation $\mc{E}=(E,E^-)$ of weight $\ep$, we define the two operators $r_{\mc{E}}$ and $l_{\mc{E}}$ on the set of $\delta$-vectors as follows:
\begin{align} \label{eq:re}	r_{\mc{E}} (\delta) &= \delta+\ep +\rank(E, \tau\delta) B; \\
	\label{eq:le} l_{\mc{E}} (\delta) &= \delta-\epc +\rank(\delta,E) B. 
\intertext{We also define the two operators $r^{\mc{E}}$ and $l^{\mc{E}}$ on the set of $\dtc$-vectors}
\notag   r^{\mc{E}} (\dtc) & =  \dtc - \tau^{-1}\ep -\rank(\tauh^{-1}\mc{E}, \dtc) B;	\\	
\notag	 l^{\mc{E}} (\dtc) &= \dtc + \tau^{-1}\epc - \rank(\tau^{-1}\dtc, \tauh^{-1}\mc{E}) B.
\end{align}	
If $\mc{E}$ is general of weight $\ep$, then we will write $\ep$ instead of $\mc{E}$ in $r_{\mc{E}}$ and $l_{\mc{E}}$.
\end{definition}
\noindent We'd like to make it clear as early as possible that the main characters of this paper are $r_\ep$ and $l_\ep$ rather than $r_{\mc{E}}$ and $l_{\mc{E}}$.
We could have defined $r_\ep$ and $l_\ep$ without mentioning $r_{\mc{E}}$ and $l_{\mc{E}}$.
It follows directly from the definition that
\begin{equation}\label{eq:r2l} r_\ep(\delta) = l_{\tau\delta}(\ep)\quad \text{and} \quad r^\ep(\dtc)=l^{\tau\delta}(\epc).
\end{equation}

Let $\b{E}$ be a lift of $\mc{E}$ in $\CQ$. We pick a general element $\b{M}$ in $\CQ$ of index $\delta$.
Let $\b{g}$ be a general morphism in $\C(\b{E},\Sigma \b{M})$.
We complete $\b{g}$ to a triangle $\Sigma^{-1}\b{E}\to \b{M}\to \b{R} \to \b{E}$, which gives a long exact sequence as in Corollary \ref{C:long}:
	$$\cdots\to \tauh^{-1} \mc{E} \xrightarrow{g_{-1}} M\to R \to E\xrightarrow{g} \tauh \mc{M}\cdots.$$ 
By Lemma \ref{L:gendes} (for void mutation) $\b{g}$ will descend to a general morphism $g\in \Hom_{J}(E, \tauh \mc{M})$.
Recall the $\delta$-vector of $\mc{R}$ is nothing but the index of $\b{R}$. So by Lemma \ref{L:index}
$$\delta_{\mc{R}} = \delta_{\mc{M}} + \delta_{\mc{E}} + \rank(g)B = \delta + \ep + \rank(E,\tau\delta)B = r_{\mc{E}}(\delta).$$
Since $\Sigma^{-1} \b{g}$ is also general in $\mc{C}(\Sigma^{-1}\b{E},\b{M})$, similarly we have that
$$\dtc_{\mc{R}} = \dtc_{\mc{M}} + \dtc_{\mc{E}} - \rank(g_{-1})B = \dtc + \epc - \rank(\tauh^{-1}\mc{E},\delta)B = r^{\mc{E}}(\dtc).$$
Note that this $r_{\mc{E}}(\delta)$ is equal to $\b{r}_{\b{E}}(\delta)$ from Definition \ref{D:rlc}. If $\mc{E}$ is general of weight $\ep$, then $r_{\mc{E}}(\delta)=r_{\ep}(\delta)$ also equals to the $\b{r}_\ep(\delta)$. We have the similar discussion for $l_{\mc{E}}(\delta)$.
We summarize in the following lemma.
\begin{lemma} \label{L:brl=rl} We have that $\b{r}_{\b{E}}(\delta) = r_{\mc{E}}(\delta)$ and $\b{l}_{\b{E}}(\delta) = l_{\mc{E}}(\delta)$;  $\b{r}_{\ep}(\delta) = r_{\ep}(\delta)$ and $\b{l}_{\ep}(\delta) = l_{\ep}(\delta)$.
\end{lemma}

Recall the direct sum notation for $\delta$-vectors in the end of Section \ref{S:GP}.
\begin{proposition}\label{P:CDP} For any two weight vectors $\ep_1$ and $\ep_2$ of $(Q,\S)$, the following are equivalent: \begin{enumerate}
		\item $\ep = \ep_1 \oplus \ep_2$;
		\item $r_\ep = r_{\ep_1} r_{\ep_2} = r_{\ep_2} r_{\ep_1}$;
		\item $l_\ep = l_{\ep_1} l_{\ep_2} = l_{\ep_2} l_{\ep_1}$.
	\end{enumerate}
\end{proposition}
\begin{proof} We only show the equivalence of (1) and (2). If $\ep = \ep_1 \oplus \ep_2$, then $\e(\ep_1,\ep_2)=\e(\ep_2,\ep_1)=0$ by \cite[Theorem 4.4]{DF}.
By Corollary \ref{C:e0}.(3) we have that $\mc{C}(\b{E}_1,\Sigma\b{E}_2)=\mc{C}(\b{E}_2,\Sigma\b{E}_1)=0$ for $\b{E}_i$ general of index $\ep_i$ ($i=1,2$).
By Lemma \ref{L:rle=0} we have that $\b{r}_{\b{E}_1}\b{r}_{\b{E}_2} = \b{r}_{\b{E}_1\oplus \b{E}_2} = \b{r}_{\b{E}_2}\b{r}_{\b{E}_1}$. By Lemma \ref{L:brl=rl} we get $r_\ep = r_{\ep_1} r_{\ep_2} = r_{\ep_2} r_{\ep_1}$.
	
Conversely, if $\ep \neq \ep_1 \oplus \ep_2$, then at least one of $\e(\ep_1,\ep_2)$ and $\e(\ep_2,\ep_1)$ is nonzero, say $\e(\ep_2,\ep_1)\neq 0$, so $\rank(\ep_1,\tau\ep_2)\neq 0$.
Let $\delta$ be the zero vector, then $r_{\ep_2}(\delta)=\ep_2$ and $r_\ep(\delta) = \ep_1+\ep_2$.
But $r_{\ep_1}(\ep_2) = \ep_1+\ep_2+\rank(\ep_1,\tau\ep_2)B$ with $\rank(\ep_1,\tau\ep_2)\neq 0$. Hence $r_\ep \neq r_{\ep_1} r_{\ep_2}$.
\end{proof}
\noindent This proposition suggest that to study $r_\ep$ and $l_\ep$ it is harmless to assume $\ep$ is {\em indecomposable}, that is, a general presentation of weight $\ep$ is indecomposable.

It is unclear if we can ask $\mc{R}$ to be general of weight $r_{\ep}(\delta)$.
However, this can be done if $r_{\ep}$ can be generically lifted to $\CQ$ at $\delta$ in the sense of Definition \ref{D:liftc}.
By abuse of language, in this case we also say $r_{\ep}$ can be generically lifted at $\delta$.

\begin{lemma}\label{L:rlmu} Suppose that the operator $r_{\ep}$ (or $l_{\ep}$) can be generically lifted at $\delta$. 
Then the operator commutes with any sequence of mutations $\mub$ and $\tau^i$:
\begin{align*} \mub(r_{\ep}(\delta)) &= r_{\mub(\ep)} (\mub(\delta)) \ &{ and }&& \ \mub(l_{\ep}(\delta)) &= l_{\mub(\ep)} (\mub(\delta)); \\
	\tau^i(r_{\ep}(\delta)) &= r_{\tau^i\ep} (\tau^i\delta) \ &{ and }&& \ \tau^i(l_{\ep}(\delta)) &= l_{\tau^i\ep} (\tau^i\delta).
\end{align*}
\end{lemma}
\begin{proof} We only prove the statement for $l_\ep$.
By definition, a general object $\b{L}$ of index $\b{l}_\ep(\delta)$ fits into a triangle $\Sigma^{-1}\b{E} \to \b{L}\to\b{M}\xrightarrow{\b{g}} \b{E}$
with the properties given in Definition \ref{D:liftc}. 
Then by Lemma \ref{L:brl=rl} $\b{l}_{\ep}(\delta) = l_{\ep}(\delta)$.
Let $\mub^-: \mc{C}_{Q,\mc{S}} \to \mc{C}_{Q',\mc{S}'}= \mc{C}_{\mub(Q,\mc{S})}$ be the triangle equivalence corresponding to the sequence of mutations.
Apply $\mub^-$ to the above triangle, and we get another triangle $\Sigma^{-1}\b{E}' \to \b{L}'\to\b{M}'\xrightarrow{\b{g}'} \b{E}'$ in $\mc{C}_{Q',\mc{S}'}$.
By Theorem \ref{T:twomu}, $\mc{M}' = \mub(\mc{M})$. Since $\mc{M}$ is general of weight $\delta$, 
$\delta_{\mc{M}'} = \mub(\delta_{\mc{M}}) = \mub(\delta)$ by Lemma \ref{L:genmu}. 
For the same reason, $\delta_{\mc{E}'} = \mub(\ep)$ and $\delta_{\mc{L}'} = \mub(l_\ep(\delta))$.
By Lemma \ref{L:gendes}, the induced map $g': M'\xrightarrow{} E'$ is general,
so $\delta_{\mc{L}'} = l_{\mub(\ep)}(\mub(\delta))$ by definition.
Hence, $\mub(l_\ep(\delta)) = l_{\mub(\ep)}(\mub(\delta))$.

For the statement about $\tau$, let us look at the long exact sequence 
\begin{equation}  \cdots \to \tauh^{-1} \mc{L}\to \tauh^{-1} \mc{M}\xrightarrow{g_{-1}} \tauh^{-1} \mc{E} \xrightarrow{} L\to M\xrightarrow{g} E \xrightarrow{} \tauh \mc{L} \to \tauh \mc{M} \xrightarrow{g_1} \tauh \mc{E} \to  \cdots. 
\end{equation}
As $\mc{L}, \mc{M}$ and $\mc{E}$ are general of weight $l_\ep(\delta),\delta$ and $\ep$, $\tau^i\mc{L}, \tau^i\mc{M}$ and $\tau^i\mc{E}$ are general of weight $\tau^il_\ep(\delta),\tau^i\delta$ and $\tau^i\ep$ by Theorem \ref{T:genpi}.
Since $\b{g}\in \C(\b{M},\b{E})$ is general, so is each $\Sigma^i \b{g}$, and so is each $g_i$ by Lemma \ref{L:gendes}.
Hence $\tau^i\mc{L}$ has weight $l_{\tau^i\ep}(\tau^i\delta)$, and the equality $\tau^i(l_\ep(\delta))=l_{\tau^i\ep}(\tau^i\delta)$ follows.
\end{proof}

\subsection{Sufficient Conditions for the Generic Lifting}
Now we explore some sufficient conditions such that $r_\ep$ and $l_\ep$ can be generically lifted.
\begin{lemma} \label{L:lift233} Any square given by a quotient presentation 
\begin{equation*} \label{eq:quop} \xymatrix{P_-^M \ar[d]_{d_M} \ar@{->>}[r]^{} & P_-^N\ar[d]_{d_N} \\
P_+^M \ar@{->>}[r]^{}  & P_+^N }
\end{equation*}	
with $\coker(d_M) = M$ and $\coker(d_N) = N$, can be lifted to a nine-diagram
	$$\xymatrix{
		\T_{-}^{{\rm L}}  \ar[r]\ar[d]_{\b{d}_{\rm L}} &\T_{-}^{{\rm M}}  \ar[r]\ar[d]_{\b{d}_{\rm M}} & \T_{-}^{{\rm N}} \ar[r]\ar[d]_{\b{d}_{\rm N}} & \Sigma\T_{-}^{{\rm L}}\ar[d]  & \\ 
		\T_{+}^{{\rm L}}  \ar[r]\ar[d] &\T_{+}^{{\rm M}}  \ar[r]\ar[d] & \T_{+}^{{\rm N}} \ar[r]\ar[d] & \Sigma\T_{+}^{{\rm L}}\ar@{-->}[d] & \\ 
		\b{L}  \ar[r] &\b{M} \ar@{-->}[r] & \b{N} \ar@{-->}[r] & \Sigma\b{L} & 
	}$$
	such that each $\b{T}_{\pm}^*\in \add \T$ for $*=\b{L},\b{M}$ or $\b{N}$, and $\T_-^{\rm L} \xrightarrow{\b{d}_{\rm L}} \T_+^{\rm L} \to \b{L}$ lifts the corresponding subpresentation as well.
\end{lemma}
\begin{proof} This lemma is a straightforward variation of \cite[Lemma 3.2]{P0}. 
The point is that the quotient presentation together with its subpresentation can be lifted to the upper two rows of the diagram. These rows are split triangles so can be completed to the nine diagram (for any $\b{d}_{\rm L}$) by \cite{N}.
\end{proof}

\begin{lemma}\label{L:rltri} If $\rank(\delta,\ep)=0$, then ${l}_\ep$ can be generically lifted at $\delta$. That is,
a general object $\b{L}$ of index $\delta+\tau^{-1}\ep$ fits into a triangle $\Sigma^{-1}\b{E}\xrightarrow{} \b{L}\xrightarrow{\b{f}} \b{M}\xrightarrow{\b{g}} \b{E}$,
where $(\b{M},\b{E})$ is a general pair of index $(\delta,\ep)$ and $\b{g}$ is a general morphism in $\mc{C}(\b{M},\b{E})$.
\end{lemma}

\begin{proof} $\rank(\delta,\ep)=0$ is equivalent to $\hom(\delta,\ep)=0$. So $\e(\tau^{-1}\ep,\delta)=0$.
By Theorem \ref{T:subp} a general presentation $d_{\mc{L}}$ in $\Hom_J(P([-\tau^{-1}\ep]_++[-\delta]_+), P([\tau^{-1}\ep]_++[\delta]_+]))$
has a quotient presentation $d_{\mc{M}}$ of weight $\delta$, and thus a subpresentation $d_1$ of weight $\tau^{-1}\ep$.
Note that the presentation $d_{\mc{L}}$ is homotopy equivalent to a general presentation of weight $l_\ep(\delta) = \delta+\tau^{-1}\ep$ by Lemma \ref{L:homotopy}.
Moreover, by Lemma \ref{L:sqgen} we may assume $d_{\mc{M}}$ and $d_{\tau^{-1}\mc{E}}$ are general as a pair, and
$\eta$ is general in $\E(d_{\mc{M}},d_{\tau^{-1}\mc{E}})$ .
We lift the quotient presentation $d_{\mc{L}} \to d_{\mc{M}}$ to a nine diagram as in Lemma \ref{L:lift233}:
$$\xymatrix{
\T_{-}^{{\rm  \Sigma^{-1}E}}\ar[r]\ar[d]_{\b{d}_{{\rm  \Sigma^{-1}E}}} & \T_{-}^{{\rm L}}  \ar[r]\ar[d]_{\b{d}_{\rm L}} & \T_{-}^{{\rm M}} \ar[r]\ar[d]_{\b{d}_{\rm M}} & \Sigma \T_{-}^{{{\rm  \Sigma^{-1}E}}}\ar[d]_{} \\ 
\T_{+}^{{\rm  \Sigma^{-1}E}}\ar[r]\ar[d] & \T_{+}^{{L}}  \ar[r]\ar[d] & \T_{+}^{{M}} \ar[r]\ar[d] & \Sigma \T_{+}^{{\rm  \Sigma^{-1}E}}\ar@{-->}[d] \\ 
\Sigma^{-1}\b{E}\ar[r] & \b{L} \ar@{-->}[r] & \b{M} \ar@{-->}[r]^{\b{g}} & \b{E} 
}$$
By Construction $\b{L}$ is general of index $l_\ep(\delta)$, $(\b{M},\b{E})$ is a general pair of index $(\delta,\ep)$.
By Lemma \ref{L:EinC} the element $\eta\in \E(d_{\mc{M}}, d_{\tau^{-1}\mc{E}})$ is now identified with a general element $\b{g}$ in $(\Sigma \T)(\b{M}, \b{E})$.
But by Corollary \ref{C:e0} every morphism $\mc{C}(\b{M}, \b{E})$ factors through $\Sigma \T$.
Hence, $\b{g}$ is in fact a general element in $\mc{C}(\b{M}, \b{E})$.
\end{proof}	

\begin{remark}\label{r:rltri} (1). Replacing $\delta$ and $\ep$ by $\ep$ and $\tau\delta$ in Lemma \ref{L:rltri}, we get the following statement:
if $\rank(\ep,\tau\delta)=0$, then $\b{r}_\ep$ can be generically lifted at $\delta$.
	
(2). With a little more effort, we can show that $\b{f}$ can be assumed to general in $\C(\b{L},\b{M})$. The proof will start with taking a general morphism $\b{f}:\b{L}\to \b{M}$,
then by Corollary \ref{C:long} $f$ is surjective. Then the rest of the proof is similar to that of Lemma \ref{L:rltri1} below.
However this by no means says that $\rank(f)$ is the general rank from $l_{\ep}(\delta)$ to $\delta$ because $(\b{L},\b{M})$ as a pair may not be in general positions.
\end{remark}

\begin{lemma}\label{L:rltri1} {\ }\begin{enumerate}
\item	If $\rank(\delta,\ep)=\dv(\ep)$, then $\b{l}_\ep$ can be generically lifted at $\delta$.
That is, a general object $\b{L}$ of index $\delta-\ep$ fits into a triangle $\b{L}\to \b{M}\xrightarrow{\b{g}} \b{E} \xrightarrow{\b{h}} \Sigma\b{L}$, where $(\b{M},\b{E})$ is a general pair of index $(\delta,\ep)$ and $\b{g}$ is a general morphism in $\mc{C}(\b{M},\b{E})$.
\item   If $\rank(\delta,\ep)=\dv(\delta)$, then $\b{l}_\ep$ can be generically lifted at $\delta$. 
\end{enumerate}
\end{lemma}
\begin{proof} (1). Take a general pair $(\b{M}, \b{E})$ and a general morphism $\b{g}:\b{M}\to \b{E}$, which descends to a general homomorphism $g:M\to E$ by Lemma \ref{L:gendes}. By the rank condition, $g$ is surjective.
We complete $\b{g}$ to a triangle $\b{L}\to \b{M}\xrightarrow{\b{g}} \b{E} \xrightarrow{\b{h}} \Sigma\b{L}$. By Lemma \ref{L:index} the index of $\b{L}$ is $\delta-\ep$. We need to show that $\b{L}$ may be assumed to be general of index $\delta-\ep$.

Let $\T_-^{\rm E}\xrightarrow{\b{d}_{\rm E}} \T_+^{\rm E} \to \b{E} \to \Sigma \T_-^{\rm E}$ and
$\T_-^{\rm L}\xrightarrow{\b{d}_{\rm L}} \T_+^{\rm L} \to \b{L} \to \Sigma \T_-^{\rm L}$ be two triangles with $\b{d}_{\rm E}\in \THom(\ep)$ and $\b{d}_{\rm L}\in \THom(\delta-\ep)$.
We set $\T_{\pm}^{\rm M} = \T_{\pm}^{\rm E} \oplus \T_{\pm}^{\rm L}$.
As $h=0$, $\b{h}$ must factor through $\Sigma\T$. Since $\C(\T,\Sigma\T)=0$, $\T_+^{\rm E} \to \b{E} \xrightarrow{\b{h}} \Sigma\b{L}$ vanishes, so $\T_+^{\rm E} \to \b{E}$ factors through $\b{M}$.
This gives the following square:
$$\xymatrix{\T_{+}^{{\rm M}}  \ar[r]\ar[d] & \T_{+}^{{\rm E}} \ar[d]   \\ 
\b{M} \ar[r]^{\b{g}} & \b{E}  & 
}$$
which can be completed to a nine-diagram
$$\xymatrix{
	\T_{-}^{\rm L} \ar[r]\ar[d]_{\b{d}_{\rm L}} &\T_{-}^{{\rm M}}  \ar[r]\ar[d]_{\b{d}_{\rm M}} & \T_{-}^{{\rm E}} \ar[r]\ar[d]_{\b{d}_{\rm E}} & \Sigma\T_{-}^{{\rm L}}\ar[d]  & \\ 
	\T_{+}^{\rm L} \ar[r]\ar[d] &\T_{+}^{{\rm M}}  \ar[r]\ar[d] & \T_{+}^{{\rm E}} \ar[r]\ar[d] & \Sigma\T_{+}^{{\rm L}}\ar[d] &   \\ 
	\b{L} \ar[r] &\b{M} \ar[r]^{\b{g}} & \b{E} \ar[r]^{} & \Sigma\b{L} & 
}$$
Apply $\C(\T,-)$ we get the following diagram
\begin{equation*}\xymatrix{
P_-^L \ar@{^(->}[r] \ar[d]_{d_{\mc{L}}} & P_-^M \ar[d]_{d_{\mc{M}}} \ar@{>>}[r]^{\pi_-} & P_-^E\ar[d]_{d_{\mc{E}}} \\
P_+^L \ar@{^(->}[r] \ar@{>>}[d]& P_+^M \ar@{>>}[r]^{\pi_+} \ar@{>>}[d]_{} & P_+^E \ar@{>>}[d]_{} \\
	L \ar[r] & M\ar@{>>}[r]^{g} & E}
\end{equation*}
In particular, we get for a pair of general presentations $d_{\mc{M}}$ and $d_{\mc{E}}$ of weight $\delta$ and $\ep$,
a general morphism from $d_{\mc{M}}$ to $d_{\mc{E}}$ is surjective.
Finally, a similar argument as in the proof of Lemma \ref{L:sqgen} shows that the kernel $d_{\mc{L}}$ of such a morphism is general as well. That is equivalent to say $\b{L}$ may be assumed to be general of index $\delta-\ep$.

(2). Note that if an object $\b{L}$ has index $\b{l}_{\ep}(\delta)=\dtc-\epc$ (by \eqref{eq:le} and \eqref{eq:delta2dual}), then $\Sigma\b{L}$ has coindex $\epc-\dtc$. So here it is convenient to work with the following obvious variant of Definition \ref{D:liftc}:
a general object $\Sigma\b{L}$ of coindex $\epc-\dtc$ fits into a triangle $\b{L}\xrightarrow{\b{f}} \b{M}\xrightarrow{\b{g}} \b{E} \to \Sigma\b{L}$,
where $(\b{M},\b{E})$ is general pair of coindex $(\dtc,\epc)$ and $\b{g}$ is a general morphism in $\mc{C}(\b{M},\b{E})$.
The proof is then similar to (1).
\end{proof}

\begin{remark} \label{r:rltri1} With little more effort, we can show that $\b{h}$ in Lemma \ref{L:rltri1}.(1) is a general morphism in $\mc{C}(\b{E},\Sigma\b{L})$, and $\b{f}$ in (2) is a general morphism in $\mc{C}(\b{L},\b{M})$. To see this, we take $\b{h}$ as an example.
Note that we do not claim that $\b{g}$ and $\b{h}$ are general as a pair, so we can start the proof all over again.
Since $\rank(\delta,\ep)=\dv(\ep)$, a general representation $M$ in $\PC(\delta)$ has a quotient representation in $\PC(\ep)$. So $\e(\delta-\ep,\ep)=0$ by Theorem \ref{T:subp}.
Then the same argument as in Lemma \ref{L:rltri} shows that $\b{h}$ can be assumed to be general.
But $\rank(h)$ may not be the general rank from $\ep$ to $\tau(l_{\ep}(\delta))$ because $(\b{E}, \Sigma\b{L})$ may not be general as a pair.
\end{remark}


\begin{definition} We say $(\delta,\ep)$ has {\em completely extremal rank} if any of the following occurs:
$$\rank(\delta,\ep)=0,\quad  \rank(\delta,\ep)=\dv(\delta),\quad \rank(\delta,\ep)=\dv(\ep).$$
\end{definition}	
	
To summarize what Lemmas \ref{L:rltri} and \ref{L:rltri1} say: if $(\delta,\ep)$ has completely extremal rank, then $l_\ep$ can be generically lifted at $\delta$.
So by Lemma \ref{L:rlmu} these operators commutes with mutations and $\tau^i$. 
The same statement also holds for $r_\ep$ and the other two operators $l^\ep$ and $r^\ep$.
We summarize together in the proposition below.
\begin{proposition} \label{P:rlmu} If the left column has the completely extremal rank, then the right column commutes with any sequence of mutations and $\tau^i$ in the sense of Lemma \ref{L:rlmu}.
	\begin{align*}  &\rank(\ep,\tau\delta)   &&  r_{\ep}(\delta),  \\
		&\rank(\delta,\ep)   &&  l_{\ep}(\delta),  \\
		&\rank(\tau^{-1}\ep,\dtc)   &&  r^{\ep}(\dtc),  \\
		&\rank(\tau^{-1}\dtc, \tau^{-1}\ep)   &&  l^{\ep}(\dtc).  \\
	\end{align*}
\end{proposition}

\begin{definition} An {\em extended mutation sequence} is a composition of ordinary mutations $\mu_u$ and the $AR$-translation $\tau$ or its inverse $\tau^{-1}$.
	We also denote $\tau$ and $\tau^{-1}$ by $\mu_+$ and $\mu_-$ respectively, though they are not involutions in general.
\end{definition}
\noindent As $\tau$ commutes with mutations (Lemma \ref{L:taucommu}), we can say the right column commutes with any extended mutation sequence in the conclusion of Proposition \ref{P:rlmu}. 
Moreover, Theorem \ref{T:twomu} can be naturally extended if we take $\Sigma$ as the lifted mutation for $\tau$.

\subsection{The Main Results}	
We have the following exact-sequence incarnation of the two operators.
\begin{theorem} \label{T:rl} Let $\ep$ and $\delta$ be two weight vectors for a Jacobi-finite quiver with potential $(Q,\S)$.
\begin{enumerate}
\item Suppose that there is an extended mutation sequence $\mub$ such that $(\mub(\ep), \mub(\tau\delta))$ has completely extremal rank $r$. Then there is an exact sequence
\begin{equation}\label{eq:rseq} \cdots \to  \tauh^{-1} \mc{M}\xrightarrow{f_{-1}} \tauh^{-1} \mc{R} \xrightarrow{g_{-1}} \tauh^{-1} \mc{E}\xrightarrow{h_{-1}} M\xrightarrow{f_0} R \xrightarrow{g_0} E \xrightarrow{h_0} \tauh \mc{M} \xrightarrow{f_1} \tauh \mc{R} \xrightarrow{g_1} \tauh \mc{E} \xrightarrow{h_1} \tauh^2 \mc{M}\to \cdots,\end{equation}	
where $\mc{R}$ is general of weight $r_\ep(\delta)$, $(\mc{M},\mc{E})$ is general as a pair of weights $\delta$ and $\ep$, and $h_i$ is a general homomorphism in $\Hom_J(\tauh^{i}\mc{E}, \tauh^{i+1} \mc{M})$.
\item Suppose that there is an extended mutation sequence $\mub$ such that $(\mub(\delta), \mub(\ep))$ has completely extremal rank $r$.  Then there is an exact sequence
\begin{equation}\label{eq:lseq} \cdots \to  \tauh^{-1} \mc{L}\xrightarrow{f_{-1}} \tauh^{-1} \mc{M} \xrightarrow{g_{-1}} \tauh^{-1} \mc{E}\xrightarrow{h_{-1}} L\xrightarrow{f_0} M \xrightarrow{g_0} E \xrightarrow{h_0} \tauh \mc{L} \xrightarrow{f_1} \tauh \mc{M} \xrightarrow{g_1} \tauh \mc{E} \xrightarrow{h_1} \tauh^2 \mc{L}\to \cdots,\end{equation}	
where $\mc{L}$ is general of weight $l_\ep(\delta)$, $(\mc{M},\mc{E})$ is general as a pair of weight $\delta$ and $\ep$, and $g_i$ is a general homomorphism in $\Hom_J(\tauh^{i} \mc{M}, \tauh^{i} \mc{E})$.
\end{enumerate}
\end{theorem}
\begin{proof} We will prove the statement (2) only because (1) can be proved in a similar fashion.	
After application of the extended mutation sequence, we are in one of the three situations of Lemmas \ref{L:rltri} and \ref{L:rltri1}. So we get a triangle $\b{L}'\xrightarrow{\b{f}'} \b{M}'\xrightarrow{\b{g}'} \b{E}' \xrightarrow{\b{h}'} \Sigma\b{L}'$ in $\C_{Q',\mc{S}'}$, which generically lift $l_{\ep'}$ at $\delta'$ where $\delta'=\mub(\delta)=\ind \b{M}'$ and $\ep'=\mub(\ep)=\ind \b{E}'$.

By Theorem \ref{T:twomu} the extended mutation sequence corresponds to a sequence of triangle equivalences. Under this equivalence, we get a triangle $\b{L}\xrightarrow{\b{f}} \b{M}\xrightarrow{\b{g}} \b{E} \xrightarrow{\b{h}} \Sigma\b{L}$ with $\b{L}$ general of index $l_\ep(\delta)$, $(\b{M},\b{E})$ is a general pair of indices $(\delta,\ep)$, and $\b{g}$ is general in $\mc{C}(\b{M},\b{E})$.
Apply the functor $F=\C(\T,-)$ to this triangle, we get the desired long exact sequence \eqref{eq:lseq} with desired generic condition by Lemma \ref{L:gendes}.
\end{proof}

\begin{remark} \label{r:genhom} (1). In fact we can also say something about $f_i$ and $h_i$ in \eqref{eq:lseq}. If $r=0$ or $\dv(\mub(\delta))$, then we may assume that $f_i$ is general in $\Hom_J(\tauh^i \mc{L}, \tauh^{i}\mc{M})$; if $r=\dv(\mub(\ep))$, then we may assume that $h_i$ is general in $\Hom_J(\tauh^i\mc{E}, \tauh^{i+1}\mc{L})$ (see Remarks \ref{r:rltri} and \ref{r:rltri1}).
As remarked there, the ranks of those morphisms may be greater than the general ranks between the corresponding principal components (see Example \ref{ex:nonpair} below).

(2). In view of Conjecture \ref{c:cer} below, the assumption of existence of such a mutation sequence may not be necessary.

\end{remark}

\begin{example} If $\delta = [\ep]_+$, then $l_\ep(\delta) = [\ep]_+-\epc+\rank([\ep]_+,\ep)B = [\ep]_+-\ep = [-\ep]_+$, and the projective presentation of $E$ is a part of the above sequence.
	Similarly we have that $r_\ep([-\ep]_+) = [\ep]_+$.
\end{example}

\begin{example}\label{ex:nonpair} In general, we cannot assume $(M,R)$ and $(R,E)$ are general pairs for (1); and cannot assume $(L,M)$ and $(E,L)$ are general pairs for (2).
Consider the $3$-arrow Kronecker quiver $\Kronthree{1}{2}$. Let $\delta=(0,1)$ and $\ep=(1,-2)$.
It is easy to check that $\ext(\delta,\ep)=0$ so $r_\ep(\delta)=\delta+\ep=(1,-1)$.
If $(R,E)$ were a general pair, then $\rank(g_0)=\rank(r_\ep(\delta),\ep)=(0,0)$, and the sequence \eqref{eq:rseq} could not be exact at $E$.
We know that a general representation $R$ of weight $(1,-1)$ has a quotient representation $E$ of weight $(1,-2)$ but it is intuitively clear that $(R,E)$ is not a general pair.
	
	
\end{example}

Combining Proposition \ref{P:rlmu} and Lemma \ref{L:rlmu}, we find a way to compute the general rank from $\delta$ to $\ep$ for nondegenerate quivers with potentials.
We will see that the method works effectively as long as one of $\delta$ and $\ep$ is reachable, and works to some extent if Conjecture \ref{c:cer} is true.
Before the invention of Derksen-Weyman-Zelevinsky's mutation, the calculation of $\rank(\delta,\ep)$ seemed completely out of reach even for acyclic quivers. 

Before we state the method, we describe a ``frozen-vertex" trick. An {\em ice quiver} is a quiver with a set of special vertices, called frozen vertices,
which are forbidden to mutate.
The extended $B$-matrix of an ice quiver $Q$ is the submatrix of the $B$-matrix of $Q$ (viewed as an ordinary quiver) given by the rows indexed by the mutable vertices.
For any quiver $Q$, we can always add to $Q$ some frozen vertices $V$ together with some arrows from $V$ to $Q_0$ such that the extended $B$-matrix of the new quiver $\wtd{Q}$ has full rank.
Moreover, for each weight vector $\delta$ of $(Q,\mc{S})$, a general representation of $(\wtd{Q},\mc{S})$ of weight $\wtd{\delta}=(\delta,0,\dots,0)$ is the extension by zeros of a general representation of weight $\delta$.
In particular, $\rank(\wtd{\delta}, \wtd{\ep}) = (\rank({\delta}, {\ep}),0,\dots,0)$.
A standard way to achieve this is that for each vertex $u\in Q_0$ we add a frozen vertex $u'$ with an arrow from $u'$ to $u$.

To find $\rank(\delta,\ep)$ we follow the three steps.\\
Step 0: If $B_Q$ has full rank then go to Step 1. Otherwise we apply the above frozen vertex trick. To ease our notation below, we will denote $\wtd{Q}$, $\wtd{\delta}$ and $\wtd{\ep}$ still by 
$Q$, $\delta$ and $\ep$.
\\
Step 1:	Find a sequence of mutations $\mub$ such that $r=\rank(\mub(\delta),\mub(\ep))$ is completely extremal. Conjecturally such a sequence of mutations always exists.\\
Step 2. According to Proposition \ref{P:rlmu} and \eqref{eq:delta2dual} we have that $$l_{\ep}(\delta) = \begin{cases} \mub^{-1}(\delta'-\epc') & \text{if $r=0$;}  \\
	\mub^{-1}(\delta'-\ep') & \text{if $r=\dv(\ep')$;} \\
	\mub^{-1}(\dtc'-\epc') & \text{if $r=\dv(\delta')$,}
\end{cases}$$
where $\delta'=\mub(\delta)$ and $\ep'=\mub(\ep)$.
Finally we have that $\rank(\delta,\ep)$ is the unique vector $r\in \mb{Z}^{Q_0}$ such that
\begin{equation}\label{eq:r} r B_Q = l_\ep(\delta)-\delta+\epc. \end{equation}
It is unclear if the sequence in Step 1 always exists.
Based on numerous examples we worked out, we conjecture the existence of such a sequence (see Conjecture \ref{c:cer}).
This method will be improved in Theorem \ref{T:murank} where a direct mutation formula is given.

\begin{conjecture} \label{c:cer} Let $(Q,\S)$ be a nondegenerate Jacobi-finite QP. For any pair $(\delta,\ep)$ of $\delta$-vectors of $(Q,\S)$, there is a sequence of mutations $\mub$ such that $(\mub(\delta),\mub(\ep))$ has completely extremal rank.
\end{conjecture}
\noindent We are not completely confident about this conjecture. But we are quite confident about a related conjecture.
\begin{conjecture} \label{c:hev} Let $(Q,\S)$ be a nondegenerate Jacobi-finite QP. For any pair $(\delta,\ep)$ of $\delta$-vectors of $(Q,\S)$, there is a sequence of mutations $\mub$ such that $(\mub(\delta),\mub(\ep))$ is either $\hom$-vanishing or $\e$-vanishing.
\end{conjecture}
\noindent Clearly $\rank(\delta,\ep)=0$ is equivalent to $\hom(\delta,\ep)=0$.
We suspect that $\rank(\delta,\ep)=\dv(\ep)$ might imply $\e(\delta,\ep)=0$ so that Conjecture \ref{c:cer} implies Conjecture \ref{c:hev}.
At least the implication holds for acyclic quivers. This follows from the fact that
$$\ext_Q(\alpha,\beta) = \ext_Q(\alpha,\beta-\gamma) = \ext_Q(\alpha-\gamma,\beta),$$
where $\gamma$ is the general rank from $\rep_\alpha(Q)$ to $\rep_\beta(Q)$.
This fact was proved in the proof of \cite[Theorem 5.4]{S}.
\begin{question} \label{q:var} For a nondegenerate Jacobi-finite QP, do we have that $\e(\delta, \ep) = \e(\delta, \check{N}_l)$ where $\check{N}_l$ is the cokernel of a general morphism $\delta\to \ep$?
\end{question}
\noindent It is easy to see that the answer is positive if $\ep$ is rigid.
It is also easy to show that $\rank(\delta,\ep)=\dv(\ep)$ implies $\e(\delta,\ep)=0$ if $\ep$ is rigid.


\begin{example} \label{ex:acyclic} Let $Q$ be the quiver $\twoone{1}{2}{3}$, and $\alpha,\beta$ be the dimension vectors $(6,9,8)$ and $(3,5,2)$.
The $B$-matrix is not of full rank so we add a frozen vertex $3'$ with an arrow from $3'$ to $3$.
Then we replace $\alpha$ by $(6,9,8,0)$ and $\beta$ by $(3,5,2,0)$.
By multiplying the Euler matrix of the quiver, we find that 
$\PC(\delta) = \rep_\alpha(Q)$ and $\PC(\epc) = \rep_\beta(Q)$
for $\delta=(6,-3,-1,0)$  and $\epc=(-7,3,2,-2)$.

We notice that $\delta$ is negative reachable by the sequence $\b{u}=(3,2,1,2,1,3)$ of mutations.
We find that $\delta'= -e_3$ and $\epc' = (1,2,-1,0)$, so $l_\ep(\delta) = \mu_{\b{u}}^{-1}(-1,-2,0,0) = (7,-4,0,0)$.
Hence the general rank from $\rep_\alpha(Q)$ to $\rep_\beta(Q)$ is $(2,3,2)$.
\end{example}

\begin{example} The following quiver cannot be mutated to an acyclic quiver
	$$\cyclicfourone{1}{2}{4}{3}{a}{b}{c}$$
We put the potential $\S=abc$.	
Let us compute the general rank from $\delta=(-1,0,-2,3)$ to $\epc=(1,4,1,-9)$.
Using an algorithm in \cite{Ft} we find $\dv(\delta)=(3,5,4,3)$ and $\dv(\epc)=(5,5,2,3)$.
As $\delta(\dv(\delta))$ and $\ep(\dv(\ep))$ are not positive, $\delta$ and $\epc$ are not rigid (in particular not extended-reachable).
However, we can apply the sequence of mutations $\b{u}=(2,1,4,2,3)$ to them such that $(\delta',\epc'):=(\mub(\delta),\mub(\epc))$ is hom-vanishing.
To see this, we find by Lemma \ref{L:gdmu} that
$\delta' = (0,3,-2,-1)$, $\epc' = (-1,-1,-1,1)$, and $\dv(\epc') = (1,0,1,1)$.
Since $[\delta']_+(\dv(\epc')) = 0$, we see that $\hom(\delta',\epc')=0$.
So $\mub(l_\ep(\delta)) = \delta' - \epc' = (1,4,-1,-2)$.
Apply the mutation backward, and we find $l_\ep(\delta) = (-3,-1,1,5)$.
Finally by \eqref{eq:r} we get $\rank(\delta,\epc) = (2,4,1,2)$.
\end{example}

If $\ep$ is extended-reachable, the rank condition in Theorem \ref{T:rl} is trivially satisfied. In addition, the rigidity of $\ep$ makes the situation particularly nice. The following two corollaries will play a crucial role in \cite{Fc}.
\begin{corollary} \label{C:rlrigid} Assume that $\ep$ is extended-reachable. Then \begin{enumerate}
	\item we may assume that both $(R,E)$ and $(E,M)$ are general pairs in the sequence
\begin{equation*}\cdots \to  \tauh^{-1} \mc{M}\xrightarrow{f_{-1}} \tauh^{-1} \mc{R} \xrightarrow{g_{-1}} \tauh^{-1} \mc{E}\xrightarrow{h_{-1}} M\xrightarrow{f_0} R \xrightarrow{g_0} E \xrightarrow{h_0} \tauh \mc{M} \xrightarrow{f_1} \tauh \mc{R} \xrightarrow{g_1} \tauh \mc{E} \xrightarrow{h_1} \tauh^2 \mc{M}\to \cdots.\end{equation*}
In particular, the rank of each $g_i$ is the general rank from $\tau^{i}(r_\ep(\delta))$ to $\tau^{i}\ep$.
	Moreover, we have that \begin{align*}
	r_{\ep}(\delta) = \delta + \epc - \rank(g_0) B\ &\text{ and }\ 
		\hom_J(r_{\ep}(\delta), \ep) = \hom_J(\delta, \ep) + \epc(\rank(g_0)),\\
	r_{\ep}(\delta) = \delta + \ep + \rank(h_0) B\ &\text{ and }\  
		\e_J(r_{\ep}(\delta),\ep) = \e_J(\delta, \ep)  - \ep(\rank(h_0)),\\
	r^{\ep}(\dtc) = \dtc - \tau^{-1}\ep - \rank(h_{-1}) B\ &\text{ and }\  
		\hom_J(\tau^{-1}\ep, r^{\ep}(\dtc)) = \hom_J(\tau^{-1}\ep, \dtc)  - \tau^{-1}\ep(\rank(h_{-1})),\\
	r^{\ep}(\dtc) = \dtc - \tau^{-1}\epc + \rank(g_{-1}) B\ &\text{ and }\  
		\ec_J(\tau^{-1}\ep, r^{\ep}(\dtc)) = \ec_J(\tau^{-1}\ep, \dtc)  + \tau^{-1}\epc(\rank(g_{-1})).	
	\end{align*}		
	\item we may assume that both $(M,E)$ and $(E,L)$ are general pairs in the sequence
\begin{equation*}\cdots \to  \tauh^{-1} \mc{L}\xrightarrow{f_{-1}} \tauh^{-1} \mc{M} \xrightarrow{g_{-1}} \tauh^{-1} \mc{E}\xrightarrow{h_{-1}} {L}\xrightarrow{f_0} M \xrightarrow{g_0} E \xrightarrow{h_0} \tauh \mc{L} \xrightarrow{f_1} \tauh \mc{M} \xrightarrow{g_1} \tauh \mc{E} \xrightarrow{h_1} \tauh^2 \mc{L}\to \cdots.\end{equation*}
In particular, the rank of each $h_i$ is the general rank from $\tau^{i}\ep$ to $\tau^{i+1}(l_\ep(\delta))$.
	Moreover, we have that \begin{align*}
	l_{\ep}(\delta) = \delta - \epc + \rank(g_0) B\ &\text{ and }\ 
		\hom_J(l_{\ep}(\delta),\ep) = \hom_J(\delta,\ep) - \epc(\rank(g_0)),\\
	l_{\ep}(\delta) = \delta - \ep - \rank(h_0) B\ &\text{ and }\  
		\e_J(l_{\ep}(\delta),\ep) = \e_J(\delta,\ep)  + \ep(\rank(h_0)),\\
	l^{\ep}(\dtc) = \dtc + \tau^{-1}\ep + \rank(h_{-1}) B\ &\text{ and }\  
		\hom_J(\tau^{-1}\ep, l^{\ep}(\dtc)) = \hom_J(\tau^{-1}\ep, \dtc)  + \tau^{-1}\ep(\rank(h_{-1})),\\
	l^{\ep}(\dtc) = \dtc + \tau^{-1}\epc - \rank(g_{-1}) B\ &\text{ and }\  
		\ec_J(\tau^{-1}\ep, l^{\ep}(\dtc)) = \ec_J(\tau^{-1}\ep, \dtc)  - \tau^{-1}\epc(\rank(g_{-1})).
	\end{align*}
\end{enumerate}
\end{corollary}

\begin{proof} We will only prove the statement for $l_\ep(\delta)$.
As $\ep$ is rigid, by Lemma \ref{L:gqiso} $E$ is essentially the only general representation in $\PC(\ep)$.
So we may assume that both $(M,E)$ and $(E,\tauh \mc{L})$ are general pairs.
In particular, the rank of each $h_i$ is the general rank from $\tau^{i}(\ep)$ to $\tau^{i+1}(l_\ep(\delta))$ (see Remark \ref{r:genhom}).
	
The equalities of the left column follow directly from the definition and the exactness of the sequence.
The equalities of the right column are all easy exercises of homological algebra. We prove the first one as an illustration.
We apply $\Hom_J(-,E)$ to the exact sequence $0\to \img(h_{-1}) \to L\to \img(f_0)\to 0$, and get
$$0\to \Hom_J(\img(f_0),E)\to \Hom_J(L,E) \to \Hom_J(\img(h_{-1}),E)=0.$$
$\Hom_J(\img(h_{-1}),E)$ vanishes because $\Hom_J(\tauh^{-1}\mc{E},E)=\Ec_J(E,\mc{E})=0$.
Thus $\Hom_J(\img(f_0),E)\cong \Hom_J(L,E)$. 
Then apply $\Hom_J(-,E)$ to the exact sequence $0\to \img(f_0) \to M\to \img(g_0)\to 0$.
We have
$$0\to \Hom_J(\img(g_0),E)\to \Hom_J(M,E) \to \Hom_J(\img(f_0),E)\to \Ec_J(\img(g_0),E)=0.$$
$\Ec_J(\img(g_0),E)$ vanishes because $\Ec_J(E,E)=0$.
So $\hom_J(l_{\ep}(\delta),E) = \hom_J(\delta,E) - \epc(\rank(g_0))$.
\end{proof}

A direct application of Proposition \ref{P:rlmu} and Lemma \ref{L:rlmu} gives the following corollary.
\begin{corollary} \label{C:rle} Let $\b{u}$ be an extended mutation sequence such that $\mub(\ep)$ is nonpositive.
	The operators $l_\ep$ and $r_\ep$ on the $\delta$-vectors of $(Q,\mc{S})$ are given by 
	\begin{align*} r_{\ep} (\delta) &= \mub^{-1} (\mub(\delta) + \mub(\ep)); \\
		l_{\ep} (\delta) &= \mub^{-1} (\mub(\delta) - \mub(\ep)).
	\end{align*}
	Let $\b{u}$ be an extended mutation sequence such that $\mub(\epc)$ is nonnegative.
	The operators $l^\ep$ and $r^\ep$ on the $\dtc$-vectors of $(Q,\mc{S})$ are given by 
	\begin{align*} r^{\ep} (\dtc) &= \mub^{-1} (\mub(\dtc) + \mub(\ep)); \\
		l^{\ep} (\dtc) &= \mub^{-1} (\mub(\dtc) - \mub(\ep)).
	\end{align*}
\end{corollary}

\noindent  Even in some trivial cases, these formulas produce amusing equalities.
\begin{proposition} We have the following equalities. \begin{enumerate}
\item $\tau(\delta+e_v+\rank(P_v,\tau\delta)B) = \tau\delta-e_v;$
\item $\tau^{-1}(\delta-e_v+\rank(\delta,I_v)B) = \tau^{-1}\delta+e_v.$
\item $\tau(\dtc-e_v-\rank(P_v, \dtc)B) = \tau\dtc+e_v;$
\item $\tau^{-1}(\dtc+e_v-\rank(\tau^{-1}\dtc, I_v)B) = \tau^{-1}\dtc-e_v.$
\end{enumerate}
\end{proposition}
\begin{proof} Consider the situation when the extended mutation sequence is just a single $\mu_+=\tau$, and $\ep=e_v$.
By Corollary \ref{C:rle} we have $r_\ep(\delta) = \tau^{-1} (\tau\delta-e_v)$.
Comparing with the original definition of $r_\ep$, we obtain the first equality.
Working with $\mu_-=\tau^{-1}$ and $\epc=e_v$ for $l_\ep$, we obtain the second equality.
Working with $\mu_+=\tau$ and $\ep=-e_v$ for $r^\ep$, we obtain the third equality.
Working with $\mu_-=\tau^{-1}$ and $\epc=-e_v$ for $l^\ep$, we obtain the last equality.
\end{proof}

\begin{proposition}\label{P:rlid} Suppose that we are in the situation of Theorem \ref{T:rl}.
The composition $l_{\ep} r_{\ep}$ is the identity if and only if $\rank(R,E)=\rank(r_{\ep}(\delta), \ep)$; 
and $r_{\ep} l_{\ep}$ is the identity if and only if $\rank(E,\tauh \mc{L})=\rank(\ep, \tau(l_{\ep}(\delta)))$.
In particular, when $\ep$ is rigid, both are identities.
\end{proposition}
\begin{proof} Using the frozen-vertex trick (after Example \ref{ex:nonpair}, we may assume the $B$-matrix of the quiver has full rank.
	By Theorem \ref{T:rl} we have that $r_\ep(\delta) = \delta+\epc-\rank(R,E)B$. Then
	$$l_{\ep}(r_\ep(\delta)) = r_{\ep}(\delta) - \epc  + \rank(r_\ep(\delta), \ep) B = \delta +(\rank(r_\ep(\delta), \ep)-\rank(R,E))B .$$
Hence, $l_{\ep} r_{\ep}$ is the identity if and only if $\rank(R,E)=\rank(r_{\ep}(\delta), \ep)$.
The other equality is proved similarly.
\end{proof}
\noindent As we have seen in Example \ref{ex:nonpair}, $\rank(R,E)$ may be different from $\rank(r_{\ep}(\delta), \ep)$.

\subsection{Remarks on Dual Operators}
In view of Corollary \ref{C:rle} one may wonder the analogous definition of operators on $\dtc$-vectors such that
\begin{align} \label{eq:rce} \rc_{\ep} (\dtc) &= \mub^{-1} (\mub(\dtc) + \mub(\epc) ) ; \\
\label{eq:lce}	\lc_{\ep} (\dtc) &= \mub^{-1} (\mub(\dtc) - \mub(\epc) ).
\end{align}
where $\b{u}$ is a sequence of mutations such that $\mub(\epc)$ is nonpositive.

\begin{definition} \label{D:rlcgrank} For any $\delta$-vector $\ep$, we define the two operators $\rc_\ep$ and $\lc_\ep$ on the set of $\dtc$-vector as follows.
\begin{align*} 	\rc_\ep (\dtc) & =  \dtc +\epc -\rank(\tau^{-1}\dtc, \epc) B;\\
	\lc_\ep (\dtc) & =  \dtc -\ep -\rank(\ep, \dtc) B,
\intertext{and}	
	\rc^\ep (\delta) &= \delta+\ep +\rank(\delta,\tau\epc) B;   \\
	\lc^\ep (\delta) & =  \delta -\epc +\rank(\tau\ep, \tau\dtc) B.	
\end{align*}
\end{definition}

\noindent If we compare this definition with Definition \ref{D:rlgrank}, we immediately get the following lemma.
\begin{lemma}\label{L:dualop} We have the following equalities 
\begin{align*}
\lc_{-\epc}(\dtc)  = r^\ep(\dtc)  \quad \text{and} \quad  \rc_{-\epc}(\dtc) = l^\ep(\dtc);\\
\lc^{-\epc}(\delta) = r_\ep(\delta)  \quad \text{and} \quad  \rc^{-\epc}(\delta) = l_\ep(\delta).
\end{align*}
where $-\epc$ is viewed as a $\delta$-vector (note that $\ker(\epc) = \tau \coker(-\epc)$).
\end{lemma}
From this lemma, one can easily see that the equalities \eqref{eq:rce} and \eqref{eq:lce} do hold for $\rc_\ep$ and $\lc_\ep$ in Definition \ref{D:rlcgrank}.
We leave it to readers to formulate the analogue of Corollary \ref{C:rle} for $\rc_\ep$ and $\lc_\ep$.

We also notice that there are more equivalent statements in Proposition \ref{P:CDP} if we allow $\rc^\ep$ and $\lc^\ep$ to join the game. For example, if $\delta=\delta_1\oplus \delta_2$, then $\tau\delta=\tau\delta_1\oplus \tau\delta_2$, so $\rc^\delta = \rc^{\delta_1}\rc^{\delta_2}$ by Lemma \ref{L:dualop}.
\begin{corollary}If $\delta=\delta_1\oplus \delta_2$, then $r_\ep(\delta) = r_{r_\ep(\delta_1)}(\delta_2)=r_{r_\ep(\delta_2)}(\delta_1)$.
\end{corollary}
\begin{proof} We have that $r_\ep(\delta) = \rc^\delta(\ep) = \rc^{\delta_1}\rc^{\delta_2}(\ep) = \rc^{\delta_1}(r_\ep(\delta_2)) = r_{r_\ep(\delta_2)}(\delta_1)$. The other equality follows from the symmetry.
\end{proof}


\section{Mutation of General Ranks} \label{S:Prop}
\subsection{A Mutation Formula}
Recall from Lemma \ref{L:rlmu} that if $r_\ep$ can be generically lifted at $\delta$, then
$r_{\mub(\ep)}(\mub(\delta))=\mub (r_\ep(\delta))$ for any extended mutation sequence, and similarly for $l_\ep$. Note that this happens when either $\ep$ or $\delta$ is reachable.
\begin{theorem} \label{T:murank} Let $\gamma_r = \rank(\ep,\tau\delta)$ and $\gamma_l = \rank(\delta,\ep)$.	We denote $\gamma_r' = \rank(\ep',\tau\delta')$ and $\gamma_l' = \rank(\delta',\ep')$, where $\delta'=\mu_u(\delta)$ and $\ep'=\mu_u(\ep)$. 
\begin{enumerate}
\item If $\mu_u (r_\ep(\delta))=r_{\ep'}(\delta')$, then
\begin{align*} 
	\gamma_r'(v) &= \begin{cases} \gamma_r(v) &  \text{for all $v\neq u$}, \\
		\gamma_r [b_u]_+ - \gamma_r(u) + [\delta(u)]_+ + [\ep(u)]_+ - [r_\ep(\delta)(u)]_+ & \text{for all $v=u$}.
	\end{cases}
\end{align*}
\item If $\mu_u (l_\ep(\delta))=l_{\ep'}(\delta')$, then \begin{align*}
\gamma_l'(v) &= \begin{cases} \gamma_l(v) & \text{for all $v\neq u$}, \\
\gamma_l [b_u]_+ - \gamma_l(u) + [\delta(u)]_+ + [-\epc(u)]_+ - [l_\ep(\delta)(u)]_+ & \text{for all $v=u$}.
\end{cases} \end{align*}
\end{enumerate}
\end{theorem}
\begin{proof} We will only prove the statement for $\gamma_l$. The other one can be proved in a similar fashion.
If $B$ is not of full rank, then we apply the frozen vertex trick (after Example \ref{ex:nonpair}).
To ease our notation, we will denote $\wtd{B}$, $\wtd{\delta}$ and $\wtd{\epc}$ still by $B$, $\delta$ and $\epc$.
We know that $\gamma_l'$ is the unique vector satisfying
\begin{equation} \label{eq:gammalmu} l_\ep(\delta)' = l_{\ep'}(\delta')  = \delta' - \epc' + \gamma_l' B'. 
\end{equation}
Recall the mutation rule for $\delta$-vectors \eqref{eq:gmu}.	
For the coordinate $u$, we have that 
$$l_\ep(\delta)'(u)=-(\delta(u)-\epc(u)+\gamma_lb_u),$$
while
\begin{align*} 
\delta'(u)-\epc'(u)+\gamma_l' b_u' = -\delta(u)+\epc(u)+\sum_{v\neq u} \gamma_l(v) (-b_{v,u}) = -\delta(u)+\epc(u)-\gamma_l b_u = l_\ep(\delta)'(u).
\end{align*}
For a coordinate $v\neq u$, we have that
\begin{align*} l_\ep(\delta)'(v) &= l_\ep(\delta)(v) +[b_{v,u}]_+l_\ep(\delta)(u) - b_{v,u} [l_\ep(\delta)(u)]_+,
\end{align*}
while
\begin{align*}
&(\delta' - \epc' + \gamma_l' B')(v) \\
=& (\delta(v) +[b_{v,u}]_+\delta(u) - b_{v,u}[\delta(u)]_+ ) - (\epc(v) + [b_{v,u}]_+\epc(u) + b_{v,u}[-\epc(u)]_+) + (\gamma_l'(u)b_{u,v}' + \sum_{w\neq u} \gamma_l(w)b_{w,v}')\\
=& l_\ep(\delta)(v) + ([b_{v,u}]_+\delta(u) - b_{v,u}[\delta(u)]_+ ) - ([b_{v,u}]_+\epc(u) + b_{v,u}[-\epc(u)]_+) \\
&+ ((-\gamma_l(u)-\gamma_l'(u))b_{u,v} + \sum_{w\neq u} \gamma_l(w)(b_{w,v}'-b_{w,v})) \\
=& l_\ep(\delta)(v)  +[b_{v,u}]_+ (l_\ep(\delta)(u) - \gamma_lb_u) + (- b_{v,u}[\delta(u)]_+ ) - (b_{v,u}[-\epc(u)]_+) \\
&+ ((-\gamma_l[b_u]_+ - [\delta(u)]_+ - [-\epc(u)]_+ + - [l_\ep(\delta)(u)]_+)b_{u,v} + \sum_{w\neq u} \gamma_l(w)(b_{w,v}'-b_{w,v})) \\
=& l_\ep(\delta)(v)  +[b_{v,u}]_+l_\ep(\delta)(u) - b_{v,u} [l_\ep(\delta)(u)]_+  \\
&- \gamma_lb_u [b_{v,u}]_+ -\gamma_l[b_u]_+b_{u,v} + \sum_{w\neq u} \gamma_l(w)([b_{w,u}]_+[b_{u,v}]_+ - [-b_{w,u}]_+[-b_{u,v}]_+))\\
=& l_\ep(\delta)'(v)  - \gamma_l (b_u [b_{v,u}]_+ + [b_u]_+b_{u,v}) + \gamma_l ([b_{u}]_+[b_{u,v}]_+  - [-b_{u}]_+[-b_{u,v}]_+) \\
=& l_\ep(\delta)'(v).
\end{align*}
We thus get the desired equality \eqref{eq:gammalmu}.
\end{proof}

\begin{remark} Similarly one can show that \begin{align*} 
\gamma_r'(u) &=	\gamma_r [-b_u]_+ - \gamma_r(u) + [-\delta(u)]_+ + [-\ep(u)]_+ - [-r_\ep(\delta)(u)]_+,
\intertext{and}
\gamma_l'(u) &=	\gamma_l [-b_u]_+ - \gamma_l(u) + [-\delta(u)]_+ + [\epc(u)]_+ - [-l_\ep(\delta)(u)]_+.
\end{align*}
So one can rewrite the formula as
	\begin{align*} \gamma_r'(u) = \begin{cases} 
		\gamma_r [b_u]_+ - \gamma_r(u) + [\delta(u)]_+ + [\ep(u)]_+ & \text{if $r_{\ep}(\delta)(u)< 0$}; \\
		\gamma_r [-b_u]_+ - \gamma_r(u) + [-\delta(u)]_+ + [-\ep(u)]_+ & \text{if $r_{\ep}(\delta)(u)\geq 0$}.
\end{cases} \end{align*}
and
	\begin{align*} \gamma_l'(u) = \begin{cases} 
		\gamma_l [b_u]_+ - \gamma_l(u) + [\delta(u)]_+ + [-\epc(u)]_+& \text{if $l_\ep(\delta(u))< 0$}; \\
		\gamma_l [-b_u]_+ - \gamma_l(u) + [-\delta(u)]_+ + [\epc(u)]_+ & \text{if $l_\ep(\delta)(u)\geq 0$},
\end{cases} \end{align*}
\end{remark}

\subsection{Some Mutation Invariants}
With these explicit formulas, we find some interesting mutation-invariants assuming $r_\ep$ and $l_\ep$ can be generically lifted at $\delta$.	
\begin{proposition} \label{P:muinv} Suppose that $l_\ep$ and $r_\ep$ commutes the mutation $\mu_u$. Let $\gamma_r = \rank(\ep,\tau\delta)$ and $\gamma_l = \rank(\delta,\ep)$.  Then the following quantities are invariant under $\mu_u$:
\begin{align*} 
h_l(\delta,\ep):=&\epc({\gamma}_l) - \hom(\delta,\ep) + \hom(l_\ep(\delta),\ep);  \\   
\check{h}_l(\delta,\ep):=&\delta(\gamma_l) - \hom(\delta,\ep) + \hom(\delta, \lc^\delta(\ep)),
\intertext{and}	
e_r(\delta,\ep):=&\ep(\gamma_r)	-\e(\delta, \ep) +\e(r_{\ep}(\delta),\ep);  \\  
\check{e}_r(\delta,\ep):=&-\delta(\gamma_r)	-\e(\delta, \ep) +\e(\delta, \rc^{\delta}(\ep)).   
\end{align*}
\end{proposition}
\begin{proof} We will only prove the mutation-invariance of $\check{h}_l(\delta,\ep)$. The rest can be proved in a similar fashion. We set $\gamma=\gamma_l$.
\begin{align*} &\delta'({\gamma}') - \delta({\gamma}) \\
	=&\delta'(u)\gamma'(u) - \delta(u)\gamma(u) + \sum_{v\neq u} (\delta'(v)\gamma(v) - \delta(v)\gamma(v))\\
	=&\delta'(u)(\gamma[b_u]_+ - \gamma(u) +[\delta(u)]_+ + [-\epc(u)]_+ - [l_\ep(\delta)(u)]_+) - \delta(u)\gamma(u)  
		+ \sum_{v\neq u} \gamma(v)(\delta'(v)- \delta(v))   & \text{\!(Theorem \ref{T:murank})}\\
		=& -\delta(u)(\gamma[b_u]_+  +[\delta(u)]_+ + [-\epc(u)]_+ - [l_\ep(\delta)(u)]_+) 
		+ \sum_{v\neq u} \gamma(v)([b_{v,u}]_+\delta(u) - b_{v,u}[\delta(u)]_+) & \text{(by \eqref{eq:gmu})} \\
		=& -\delta(u)([\delta(u)]_+ + [-\epc(u)]_+ - [l_\ep(\delta)(u)]_+)
		+ \sum_{v\neq u} \gamma(v)(- b_{v,u}[\delta(u)]_+) \\
		=& -[\delta(u)]_+(\delta(u)+ \gamma b_u) -\delta(u)([-\epc(u)]_+ - [l_\ep(\delta)(u)]_+) \\
		=& -[\delta(u)]_+ l_\ep(\delta)(u) -[\delta(u)]_+\epc(u) -\delta(u)[-\epc(u)]_+ + \delta(u)[l_\ep(\delta)(u)]_+ \\
		=& [\delta(u)]_+ [-l_\ep(\delta)(u)]_+ -[-\delta(u)]_+ [l_\ep(\delta)(u)]_+ - [\delta(u)]_+[\epc(u)]_+ + [-\delta(u)]_+[-\epc(u)].
	\end{align*}
Comparing this with Lemma \ref{L:HEmu}, we find that 
\begin{align*} \delta'({\gamma}') - \delta({\gamma}) =&-(\hom(\delta',\ep') - \hom(\delta,\ep)) + (\hom(\delta', \tau l_\ep(\delta)') - \hom(\delta, \tau l_\ep(\delta))    &  \\
	=&-(\hom(\delta',\ep') - \hom(\delta,\ep)) + (\hom(\delta',\lc^\delta(\ep)') - \hom(\delta,\lc^\delta(\ep))    & \text{(by \eqref{eq:r2l} and Lemma \ref{L:dualop}).}
\end{align*}
\end{proof}

\begin{remark} If $\delta$ is extended-reachable, then the above invariants always vanish as seen in Corollary \ref{C:rlrigid}.
If $\delta$ is not rigid, they may not be zero (but conjecturally always nonnegative).
The mathematical meaning of this invariant is still mysterious to us.
\end{remark}

\section{Appendix: Conjecture \ref{c:hev} implies the Duality Pairing and the Saturation}
Recall from \cite{Ft} that the {\em tropical $F$-polynomial} $f_M$ of a representation $M$ is the function $(\mb{Z}^{Q_0})^* \to \mb{Z}_{\geq 0}$ defined by
	$$\delta \mapsto \max_{L\hookrightarrow M}{\delta(\dv L)};$$
	The {\em dual} tropical $F$-polynomial $\fc_M$ of a representation $M$ is the function $(\mb{Z}^{Q_0})^* \to \mb{Z}_{\geq 0}$ defined by
	$$\delta \mapsto \max_{M\twoheadrightarrow N}{\delta(\dv N)},$$
where $\delta$ is viewed as an element in $(\mb{Z}^{Q_0})^*$ via the usual dot product. Clearly $f_M$ and $\fc_M$ are related by $f_M(\delta)-\fc_M(-\delta)= \delta(\dv M)$.
If $M$ is general of weight $\delta$, then we will write $f_{\delta}$ for $f_M$.
Similarly we can define $\fc_\delta$ and $f_{\dtc}$.

In \cite[Section 6]{Ft} we showed that Fock-Goncharov's duality pairing conjecture for generic bases of skew-symmetric cluster algebras can be deduced from the following conjecture: for any pair $(\delta,\epc)$ we have that 
$$f_{\epc}(\delta) = \fc_{\delta}(\epc).$$
In fact, they are equivalent if the $B$-matrix of $Q$ has full rank.
We refer readers to \cite[Section 12]{FG} and \cite[Section 6]{Ft} for the details of the duality pairing conjecture. 
As mentioned in \cite{Ft}, a more optimistic conjecture is the following {\em generic pairing} \begin{equation} \label{eq:genpairing} f_{\epc}(\delta) =\hom(\delta,\epc) = \fc_{\delta}(\epc). \end{equation}
The following theorem generalizes \cite[Theorem 3.22]{Ft}.
\begin{theorem} \label{T:dualitypairing} Assume either of the following two situations:
\begin{enumerate}
\item there is a sequence of mutations $\mub$ such that $(\mub(\delta), \mub(\epc))$ is $\hom$-vanishing;
\item there are two sequences of mutations $\mub$ and $\mu_{\check{\b{u}}}$ such that $(\mub(\delta), \mub(\epc))$ is  $\e$-vanishing and $(\mu_{\check{\b{u}}}(\delta), \mu_{\check{\b{u}}}(\epc))$ is $\ec$-vanishing.
\end{enumerate}
Then $f_{\epc}(\delta)=\hom(\delta,\epc)=\fc_{\delta}(\epc)$.
\end{theorem}
\begin{proof} Let $\delta'=\mub(\delta)$ and $\epc'=\mub(\epc)$.
	If $\hom(\delta', \epc')=0$, then $f_{\epc'}(\delta') \geq \hom(\delta', \epc')$.
But $f_{\epc'}(\delta') \leq \hom(\delta', \epc')$ by \cite[Lemma 3.5]{Ft}.
Hence $f_{\epc'}(\delta')=\hom(\delta', \epc')$.
Then by \cite[Lemma 3.21]{Ft} we get $f_{\epc}(\delta)=\hom(\delta,\epc)$.
Similarly we can show that $\fc_{\delta}(\epc)=\hom(\delta,\epc)$.

If $\e(\delta', \epc')=0$, then a similar argument shows that $\fc_{\epc}(-\delta)=\e(\delta,\epc)$,
which is equivalent to $f_{\epc}(\delta)=\hom(\delta,\epc)$.
Similarly, from $\ec(\delta', \epc')=0$ we can conclude that $\ec(\delta,\epc)=f_{\delta}(-\epc)$, which is equivalent to $\hom(\delta,\epc)=\fc_{\delta}(\epc)$.
\end{proof}
\noindent Note that if $f_{\epc}(\delta)=\hom(\delta,\epc)$, then by \cite[Theorem 3.6]{Ft} $f_{\epc}(m\delta)=\hom(m\delta,\epc)$ for any $m\in\mb{N}$; similarly if $\fc_{\delta}(\epc)=\hom(\delta,\epc)$, then $\fc_{\delta}(m\epc)=\hom(\delta,m\epc)$ for any $m\in\mb{N}$.
Also note that $f_{\epc}(m\delta) = m f_{\epc}(\delta)$.
So the conclusion of Theorem \ref{T:dualitypairing} implies that
\begin{equation}\label{eq:homflu} \hom(m\delta,\epc)=\hom(\delta, m\epc)=m\hom(\delta,\epc) \quad \text{for any $m\in\mb{N}$}.
\end{equation}
Conversely, recall from \cite[Theorem 3.22]{Ft} that we always have that $$f_{\epc}(m\delta)=\hom(m\delta,\epc)\ \text{ and }\ f_{\delta}(n\epc)=\hom(\delta,n\epc)\ \text{ for some $m,n\in\mb{N}$}. $$
So if \eqref{eq:homflu} holds for the pair $(\delta,\epc)$, then we have that 
$f_{\epc}(\delta)=\hom(\delta,\epc)$ and $\hom(\delta,\epc)=\fc_{\delta}(\epc)$. 

\begin{definition} \label{D:saturation} We say that a pair $(\delta,\epc)$ has the {\em $\hom$-fluent property} if \eqref{eq:homflu} holds. 
We say that a pair $(\delta,\epc)$ has the {\em saturation property} if $\hom(\delta,\epc)=0$ whenever $\hom(m\delta,n\epc)=0$ for some $m,n\in\mb{N}$.
We say a QP is {\em $\hom$-fluent} (resp. saturated) if the $\hom$-fluent (resp. saturation) property holds for any pair $(\delta,\epc)$.
\end{definition}
\noindent The saturation property generalizes the ordinary saturation property for acyclic quivers \cite{DW1}.
The saturation property seems weaker than the $\hom$-fluent property.
But the proof of \cite[Theorem 3.22]{Ft} shows that they are in fact equivalent.
We conclude that
\begin{proposition}\label{P:saturation} The following are equivalent:
\begin{enumerate}
\item The generic pairing \eqref{eq:genpairing} holds for $(\delta,\epc)$.
\item The $\hom$-fluent property holds for $(\delta,\epc)$.
\item The saturation property holds for $(\delta,\epc)$.
\end{enumerate}
\end{proposition}

\begin{conjecture}[Saturation for nondegenerate QPs]  Every nondegenerate Jacobi-finite QP is saturated.
\end{conjecture} 
\noindent It is not hard to give a counterexample for degenerate QPs. 
\begin{corollary} Conjecture \ref{c:hev} implies both the saturation conjecture
	and the generic pairing conjecture for nondegenerate QPs.
\end{corollary}
\begin{proof} By the above discussion, it suffices to show that $f_{\epc}(\delta)=\hom(\delta,\epc)=\fc_{\delta}(\epc)$ for any pair $(\delta,\ep)$.
For any pair $(\delta,\ep)$, we also consider the pair $(\tau^{-1}\ep,\delta)$.
By Conjecture \ref{c:hev}, there is a sequence of mutations $\mub$ such that $(\mub(\tau^{-1}\ep), \mub(\delta))=(\tau^{-1}\mub(\ep), \mub(\delta))$ is either $\hom$-vanishing or $\e$-vanishing, which is equivalent to say $(\mub(\ep), \mub(\delta))$ is either $\ec$-vanishing or $\hom$-vanishing by Lemma \ref{L:H2E}.
If it is $\hom$-vanishing, then we are done by Theorem \ref{T:dualitypairing}. Otherwise, it is $\ec$-vanishing and if working with the original pair $(\delta,\ep)$ there is another sequence of mutations $\mub'$ such that $(\mub'(\ep), \mub'(\delta))$ is $\e$-vanishing, and we are done as well by Theorem \ref{T:dualitypairing}.
\end{proof}

\section*{Acknowledgement}
I would like to thank the anonymous referee for many good suggestions and corrections.

%

\bibliographystyle{amsplain}

\end{document}